\documentclass [11pt, article, oneside, letter]{memoir}
\usepackage[T1]{fontenc}
\usepackage[utf8]{inputenc}

\usepackage{hyphenat}
\usepackage{xcolor}
\usepackage[all]{xy}
\usepackage{amsfonts}
\RequirePackage{amsmath,amsthm,paralist,bm,amsthm,graphicx,color,mathrsfs,xcolor,fouridx}
\usepackage[colorlinks=true,linkcolor=black,citecolor=blue]{hyperref}

\newcommand{\aquatre}{%
\stockaiv\pageaiv%
\settypeblocksize{22cm}{13cm}{*}%
\setlrmargins{4cm}{*}{1}%
\setmarginnotes{0pt}{0pt}{0pt}%
\setulmargins{3.5cm}{*}{1}%
\setheadfoot{3\baselineskip}{3\baselineskip}%
\setheaderspaces{2\baselineskip}{*}{1}%
\checkandfixthelayout%
}
\aquatre

\newcommand{\addpoint}[1]{#1\ ---\ }

\setsecnumdepth{subsubsection}

\setsecnumformat{\csname the#1\endcsname---}
\setsubsubsechook{\setsecnumformat{{\normalfont\sffamily\csname
      the##1\endcsname\ }}}
\setsubsechook{\setsecnumformat{\csname the##1\endcsname\ ---\ }}
\setsechook{\setsecnumformat{\csname the##1\endcsname---}}

\setsecheadstyle{\centering\Large\bfseries\sffamily}
\setsubsecheadstyle{\large\bfseries\sffamily}
\setsubsubsecheadstyle{\bfseries\itshape\addpoint}
\setsubsubsecindent{1em}
\setbeforesubsubsecskip{.5em plus .2em minus -.1em}
\setaftersubsubsecskip{-0em}
\setsubparaheadstyle{\itshape}
\newtheoremstyle{thm}%
     {1.5ex plus .3ex minus .1ex}%
     {1ex plus .3ex minus .1ex}%
     {\itshape}%
     {}%
     {\sffamily}%
     {---}%
     {0em}%
     {$\bullet$\hbox{\ }#1\hbox{\ }#2}%

\theoremstyle{thm}

\newtheorem{definition}{Definition}[section]

\newtheorem{theorem}[definition]{Theorem}

\newtheorem{lemma}[definition]{Lemma}
\newtheorem{proposition}[definition]{Proposition}
\newtheorem{corollary}[definition]{Corollary}

\newtheoremstyle{note}%
     {1ex plus .3ex minus .1ex}%
     {1ex plus .3ex minus .1ex}%
     {}%
     {}%
     {\itshape}%
     {.}%
     {1em}%
     {}%
\theoremstyle{note}
\newtheorem{example}[definition]{Example}
\newtheorem{remark}[definition]{Remark}

\newlength{\remaining}

\newcommand{\labell}[1]{\POS[]\POS!L\drop{\strut\llap{$#1$}\hspace{2ex}}}
\newcommand{\labelu}[1]{\POS[]\POS!U\drop{\raisebox{1.2em}{\strut$#1$}}}
\newcommand{\labeldr}[1]{\POS[]\POS!D(2)!R(1.5)\drop{\strut\rlap{$#1$}\hspace{2ex}}}
\newcommand{\labelr}[1]{\POS[]\POS!R\drop{\strut\hspace{2ex}\rlap{$#1$}}}
\newcommand{\labeld}[1]{\POS[]\POS!D(3)\drop{\strut\rlap{$#1$}\hspace{2ex}}}
\newcommand{\verticale}{\ar@{--}[d]}

\newsavebox{\labelz}\newsavebox{\labelun}
\savebox{\labelz}{$\xymatrix@M=0.2ex{*++[F]{0}}$}
\savebox{\labelun}{$\xymatrix@M=0.2ex{*++[F]{1}}$}

\def\rebar#1{\expandafter\def\csname #1bar\endcsname{\overline{\csname
      #1\endcsname}}}		%

\newcommand{\M}{{\mathcal{M}}}

\newcommand{\N}{{\mathcal{N}}}
\newcommand{\X}{{\mathcal{X}}}

\newcommand{\bbR}{\mathbb{R}}
\newcommand{\bbZ}{\mathbb{Z}}
\renewcommand{\P}{\mathcal{P}}

\rebar{M}	%
\rebar{N}	%

\newcommand{\C}{\mathscr{C}}
\newcommand{\D}{\mathscr{D}}
\newcommand{\Dstar}{\mathfrak{D}}

\newcommand{\DCS}{{\text{\normalfont\sffamily DCS}}}

\newcommand{\up}[1]{\,\uparrow #1}

\newcommand{\Cstar}{\mathfrak{C}}

\newcommand{\slgb}{\mbox{$\sigma$-al}\-ge\-bra}

\newcommand{\BM}{\partial\M}

\newcommand{\lub}{\textsl{lub}}
\newcommand{\glb}{\textsl{glb}}

\newcommand{\vd}{\varepsilon}

\newcommand{\tq}{\;\big|\;}
\newcommand{\tqs}{\;:\;}

\newcommand{\ie}{\textsl{i.e.}}

\newcommand{\rest}[1]{\bigl|_{#1}}

\newcommand{\un}{\mathbf{1}}

\newcommand{\height}{\tau}

\numberwithin{equation}{section}

\newsavebox{\framedzero}\savebox{\framedzero}{\raisebox{1ex}{\xymatrix{*+[F]{0}}}}

\newlength{\tempa}

\newlength{\tempcc}
\newlength{\hauttrans}
\newlength{\longtrans}
\newlength{\longplace}

\newcommand{\petriexec}{%
\setlength{\hauttrans}{.66ex}
\setlength{\longtrans}{6ex}
\setlength{\longplace}{4ex}
\settoheight{\tempcc}{\strut}
\settodepth{\tempa}{\strut}
\addtolength{\tempcc}{\tempa}
\setlength{\tempcc}{2\tempcc}
\addtolength{\tempcc}{\hauttrans}}

\newcommand{\transh}{*+[F]\txt{\vbox to\hauttrans{\hbox to\longtrans{}}}}
\newcommand{\transhp}{*+[F--]\txt{\vbox to\hauttrans{\hbox to\longtrans{}}}}
\newcommand{\transv}{*+[F]\txt{\vbox to\longtrans{\hbox to\hauttrans{}}}}
\newcommand{\transvp}{*+[F--]\txt{\vbox to\longtrans{\hbox to\hauttrans{}}}}
\newcommand{\place}{*+[o][F]{\vbox to\longplace{\hbox to\longplace{}}}}

\newcommand{\transht}{*+[F.]\txt{\vbox to\hauttrans{\hbox to\longtrans{}}}}
\newcommand{\placet}{*+[o][F.]{\vbox to\longplace{\hbox to\longplace{}}}}

\newcommand{\markp}{\POS[]\drop{\bullet}}

\petriexec

\begin{document}

\begin{center}
  {\huge\bfseries
Deterministic concurrent systems
    }

\Large
\bigskip
Samy Abbes (\texttt{abbes@irif.fr})\\
{\normalsize Universit\'e de Paris --- IRIF (CNRS UMR 8243)}\\[1em]
August 2020
\end{center}

\bigskip

\begin{abstract}
We introduce deterministic concurrent systems as a subclass of concurrent systems. Deterministic concurrent system are ``locally commutative'' concurrent systems. We prove that irreducible and deterministic concurrent systems have unique probabilistic dynamics, and we characterise these systems by means of their combinatorial properties.
\end{abstract}

\section{Introduction}
\label{sec:introduction}

Trace monoids are well known models of concurrency. They represent systems able to perform several types of actions, represented by letters in a given alphabet, and with the feature that some actions may occur concurrently. If $a$ and $b$ are two concurrent actions, then the system does not distinguish between the two sequences of actions $a$-then-$b$ and $b$-then-$a$. Instead, a unique compound action $a\cdot b=b\cdot a$ may be performed. This feature is typically used when one wishes to work on the logical order between actions rather than on the chronological order.

Mathematically, a trace monoid $\M$ is a monoid generated by an alphabet~$\Sigma$, and with relations of the form $ab=ba$ for some fixed pairs of letters $(a,b)\in\Sigma\times\Sigma$. The identity $ab=ba$ in $\M$ renders the concurrency of the two actions $a$ and~$b$. 

The use of trace monoids in concurrency theory goes back at least to the 1980's with survey works such as~\cite{diekert90,diekert95}. Trace monoids had also been studied in Combinatorics under different names, as free partially commutative monoids and heaps of pieces in the seminal works \cite{cartier69} and \cite{viennot86} respectively. Hence, trace monoids stand at a junction point between computer science and combinatorics.

Despite their successful use as models of concurrency for databases for instance, trace monoids lack an essential feature present in most real-life systems, namely they lack a notion of state. Indeed, any action can be performed at any time when considering a trace monoid model; whereas, in real-life systems, some actions may only be enabled when the system enters some specified state, and then one expects the system to enter a new state, determined by the former state and by the action performed.

A natural model combining both the ``built-in'' concurrency feature of trace monoids and the notion of state arises when considering a partially defined monoid action of a trace monoid $\M$ on a finite set of states~$X$. Equivalently, instead of considering that the monoid action is only partially defined, it is more convenient  to introduce a sink state $\bot$ and to consider a total monoid action $(X\cup\{\bot\})\times\M\to(X\cup\{\bot\})$. Hence, if the system is in state~$\alpha$, performing the letter $a\in\Sigma$ brings the system into the new state~$\alpha\cdot a$, with the convention that $a$ was actually not allowed if $\alpha\cdot a=\bot$. This notion of concurrent system, introduced in~\cite{abbes19:_markov}, encompasses in particular popular models of concurrency such as bounded Petri nets~\cite{reisig85,nielsen81}.

In the present paper, we use some results previously obtained in~\cite{abbes19:_markov,abbes20} in order to study a particular case of concurrent systems, namely the class of \emph{deterministic concurrent systems}. Intuitively, a deterministic concurrent system (\DCS) is a concurrent system where no conflict between different actions can ever arise. Hence the only non-determinism left results solely from the concurrency of the model, combined with the constraints imposed by the monoid action. Deterministic concurrent systems can be related, for instance, to causal nets and to elementary event structures found in 1980's papers~\cite{nielsen81}. We prove in particular that deterministic concurrent systems correspond to concurrent systems which are ``locally commutative''.

Compared to general concurrent systems, deterministic concurrent systems appear as limit cases. For instance, we prove that their space of maximal executions is at most countable---whereas it is uncountable in general; if the system is moreover irreducible, we prove that it carries a unique probabilistic dynamics---whereas there is a continuum of them in general. Yet, proving these properties is not trivial. The definition of \DCS\ is formulated in elementary terms; their specific properties are formulated in elementary terms; but the proof of these properties relies on the combinatorics of partially ordered sets.

Beside the general properties of deterministic concurrent systems, our main contribution is to give several equivalent characterisations of concurrent systems which are both deterministic and irreducible: an algebraic characterisation; a probabilistic characterisation; a characterisation from the Analytic combinatorics viewpoint; and a characterisation through set-theoretic properties of the set of infinite executions. The multiplicity of these viewpoints suggests that the notion is worth exploring it.

Another contribution is a generalisation of the well known fact that commutative free monoids have a polynomial growth. The property that we obtain in Corollary~\ref{cor:1} is general enough to be of interest \emph{per se}.

\medskip
Although quite specific, the class of deterministic concurrent systems has a non trivial modelisation power. We also believe that understanding deterministic concurrent systems is useful for the deeper understanding of general concurrent systems.

\paragraph*{Organisation of the paper.}
\label{sec:organisation-paper}

Section~\ref{sec:preliminaries} is devoted to preliminaries, and is divided into three subsections. Sections~\ref{sec:trace-monoids-their} and~\ref{sec:conc-syst-their} survey respectively basic notions on trace monoids and on concurrent systems; Section~\ref{sec:an-elem-comp} is devoted to an elementary, yet original result of trace theory, that we tried to formulate in a way not too specific so that it could be of general interest, and that will be used later in the paper. Deterministic concurrent systems are introduced in Section~\ref{sec:determ-conc-syst}. Section~\ref{sec:irred-determ-conc} is devoted to the study of concurrent systems which are both deterministic and irreducible.

\section{Preliminaries}
\label{sec:preliminaries}

\subsection{Trace monoids and their combinatorics}
\label{sec:trace-monoids-their}

The background material introduced in this section is standard, see for instance~\cite{diekert90,diekert95}, excepted for the probabilistic notions which are borrowed from~\cite{abbes15}.

\paragraph*{Independence and dependence pairs.}
An \emph{alphabet} is a finite set, which we usually denote by~$\Sigma$, the elements of which are called \emph{letters}. An  \emph{independence pair} is a pair $(\Sigma,I)$, where $I$ is a binary symmetric and irreflexive relation on~$\Sigma$, called an \emph{independence relation}. A \emph{dependence pair} is a pair $(\Sigma,D)$, where $D$ is a binary symmetric and reflexive relation on~$\Sigma$, called a \emph{dependence relation}. With $\Sigma$ fixed, dependence and independence relations correspond bijectively to each others, through the association $D=(\Sigma\times\Sigma)\setminus I$.

\medskip{\itshape
In the remaining of Section~\ref{sec:trace-monoids-their}, we fix an independence pair $(\Sigma,I)$, with corresponding dependence pair $(\Sigma,D)$.}

\paragraph*{Traces.}
The \emph{trace monoid}\footnote{In the literature, trace monoids are also called free partially commutative monoids, and they also correspond to right-angled Artin-Tits monoids.} $\M(\Sigma,I)$ is the presented monoid $\M=\langle\Sigma\tq ab=ba\text{ for $(a,b)\in I$}\rangle$. Elements of $\M$ are called \emph{traces}. The unit element, also called \emph{empty trace}, is denoted by~$\vd$, and the concatenation of $x,y\in\M$ is denoted by~$x\cdot y$. We identify letters of the alphabet with their images in~$\M$ through the canonical mappings $\Sigma\to\Sigma^*\to\M$.

The trace monoid $\M$ is \emph{irreducible} if the dependence pair $(\Sigma,D)$, seen as a graph, is connected.

% The free monoid $\Sigma^*$ corresponds to $\M(\Sigma,I)$ with $I=\emptyset$ and the free commutative monoid on $\Sigma$ correspond to $\M(\Sigma,I)$ with $I=(\Sigma\times\Sigma)\setminus\Delta$ where $\Delta$ is the identity relation on~$\Sigma$.

\paragraph*{Length. Occurrence of letters.}
Every trace $x\in\M$ corresponds to the congruence class of some word $u\in\Sigma^*$. The \emph{length} of~$x$, denoted by~$|x|$, is the length of~$u$. For each letter $a\in\Sigma$, we write $a\in x$ whenever $a$ has at least one occurrence in~$u$, and we write $a\notin x$ otherwise.

\paragraph*{Divisibility order.}
The preorder $(\M,\leq)$ inherited from the left divisibility in $\M$ is defined by: $x\leq y\iff(\exists z\in\M\quad y=x\cdot z)$. This preorder is actually a partial order. If $x\leq y$, the element $z\in\M$ such that $y=x\cdot z$ is unique since trace monoids are left cancelable. We denote this element by $z=x\backslash y$.

\paragraph*{Cliques.}
A \emph{clique} of $\M$ is a trace of the form $x=a_1\cdot\ldots\cdot a_i$, where all $a_i$s are letters such that $i\neq j\implies (a_i,a_j)\in I$. Since all $a_i$s commute with each other, we identify the clique $x\in\M$ with the subset $\{a_1,\ldots,a_i\}\in\P(\Sigma)$. If $\C$ denotes the set of cliques of~$\M$, the restricted partial order $(\C,\leq)$ corresponds to a sub-partial order of $(\P(\Sigma),\subseteq)$. We note that $\C$ is always downward closed in $(\P(\Sigma),\subseteq)$, and that $\C$ corresponds to the full powerset $\P(\Sigma)$ if and only if $\M$ is the free commutative monoid on~$\Sigma$.

A \emph{non empty clique} is a clique $x\neq\vd$. The set of non empty cliques of~$\M$ is denoted by~$\Cstar$. Minimal elements of $(\Cstar,\leq)$ correspond to the letters of~$\Sigma$.

\paragraph*{Parallel cliques. Lower and upper bounds.}
\label{sec:lower-upper-bounds}

Any two traces $x,y\in\M$ have a greatest lower bound (\glb) in $(\M,\leq)$, which we denote by $x\wedge y$. They have a least upper bound (\lub) in $(\M,\leq)$, denoted by $x\vee y$ if it exists, if and only if they have a common upper bound.

If $x$ and $y$ are cliques, then $x\wedge y$~is the clique corresponding to the subset $x\cap y\in\P(\Sigma)$. We say that $x$ and $y$ are \emph{parallel}, denoted by $x\parallel y$, if $x\times y\subseteq I$, where $x$ and $y$ are seen as subsets of~$\Sigma$. In this case, $x\vee y$ exists and is given by $x\vee y=x\cdot y=y\cdot x$.

\paragraph*{Normal sequences.}
\label{sec:normal-sequences}

A pair $(x,y)\in\C\times \C$ is a \emph{normal pair} if: $\forall b\in y\quad\exists a\in x\quad(a,b)\in D$. This relation is denoted by $x\to y$. A sequence $(c_i)_i$ of cliques, the sequence being either finite or infinite, is a \emph{normal sequence} if $(c_i,c_{i+1})$ is a normal pair for all pairs of indices $(i,i+1)$.

Note that the empty clique satisfies $x\to\vd$ for all $x\in\C$, and $\vd\to x$ if and only if $x=\vd$.

\paragraph*{Normal form and generalised normal form.}
\label{sec:normal-form}

\cite{cartier69} For any trace $x\neq\vd$, there exists a unique integer $k\geq1$ and a unique normal sequence $(c_1,\ldots,c_k)$ of non empty cliques such that $x=c_1\cdot\ldots\cdot c_k$. The sequence $(c_1,\ldots,c_k)$ is the \emph{Cartier-Foata normal form of~$x$}, or the \emph{normal form of $x$} for short. The integer $k$ is the \emph{height} of~$x$, denoted by $k=\height(x)$.

The \emph{generalised normal form of $x$} is the infinite normal sequence $(c_i)_{i\geq 1}$ defined by $c_i=\vd$ for $i>k$. By definition, the generalised normal form of $\vd$ is the normal sequence $(\vd,\vd,\ldots)$.

For every integer $i\geq1$, we introduce the mapping $C_i:\M\to\C$ defined by $C_i(x)=c_i$, where $(c_i)_{i\geq1}$ is the generalised normal form of~$x$.

\paragraph*{Generalised traces and infinite traces.}
\label{sec:generalised-traces}

% Thanks to the existence and uniqueness of the normal form, we see that traces and finite normal sequences are essentially the same objects. One drawback is that the monoid concatenation does not easily translate in the language of normal forms. But normal forms are powerful for studying the combinatorics of traces, as we shall see later. Normal forms are also a natural way to introduce infinite traces, as follows.

A \emph{generalised trace} is any infinite normal sequence $\xi=(c_i)_{i\geq1}$ of cliques. If $c_i=\vd$ for some integer~$i$, then $c_j=\vd$ for all $j\geq i$, and then $\xi$ is the generalised normal form of a unique element of~$\M$. If $c_i\neq\vd$ for all $i\geq1$, then $\xi$ is said to be an \emph{infinite trace}.

We denote by $\Mbar$ the set of generalised traces, and by $\BM$ the set of infinite traces---the latter set is called the boundary at infinity of~$\M$. We note that $\BM$ is non empty as soon as~$\Sigma\neq\emptyset$.

We define a partial order on $(\Mbar,\leq)$ by putting, for $\xi=(c_i)_{i\geq1}$ and $\zeta=(d_i)_{i\geq1}$ two generalised traces:
\begin{gather*}
  \xi\leq\zeta\iff(\forall i\geq 1\quad c_i\leq  d_i).
\end{gather*}

The injection $\M\to\Mbar$ induces an embedding of partial orders $(\M,\leq)\to(\Mbar,\leq)$, so we simply identify $\M$ with its image in~$\Mbar$. With this identification, we have  $\Mbar=\M+\BM$, where `$+$' denotes the disjoint union. 

The family of mappings $(C_i)_{i\geq1}$ extends in the obvious way to the natural projections $C_i:\Mbar\to \C$, with restrictions $C_i:\BM\to\Cstar$.

The digraph $(\C,\to)$ is called the \emph{digraph of cliques} of the monoid. Generalised traces correspond bijectively to infinite paths in $(\C,\to)$, with finite traces corresponding to paths hitting the empty clique~$\vd$, and infinite traces corresponding to paths never hitting the empty clique.

\paragraph*{M\"obius transform.}
\label{sec:mobius-transform}

Let $f:\C\to A$ be a function where $A$ is any commutative group. The \emph{M\"obius transform} \cite{rota64} of~$f$ is the function $h:\C\to A$ defined by:
\begin{gather}
  \label{eq:1}
  \forall c\in\C\quad h(c)=\sum_{c'\in \C\tqs c\leq c'}(-1)^{|c'|-|c|}f(c').
\end{gather} 

The function $f$ can be retrieved from $h$ thanks to the \emph{M\"obius inversion formula}, which is a kind of generalised inclusion-exclusion formula:
\begin{gather}
  \label{eq:2}
  \forall c\in\C\quad f(c)=\sum_{c'\in\C\tqs c\leq c'}h(c').
\end{gather}
In particular, one has:
\begin{gather}
  \label{eq:3}
  f(\vd)=\sum_{c\in \C}h(c).
\end{gather}

\paragraph*{Valuations and probabilistic valuations.}
\label{sec:valuations}

\cite{abbes15} A \emph{valuation} is a monoid homomorphism $f:(\M,\cdot)\to(\bbR_{\geq0},\times)$. One instance is the constant valuation $f=1$. More generally, any assignation of non negative numbers $\lambda_a$ to letters $a$ of $\Sigma$ yields a valuation~$f$, obviously unique, such that $f(a)=\lambda_a$ for $a\in\Sigma$.

Let $h:\C\to\bbR$ be the M\"obius transform of a valuation~$f$, restricted to~$\C$. Then $f$ is a \emph{probabilistic valuation} whenever:
\begin{gather}
  \label{eq:7}
\bigl(  h(\vd)=0\bigr)\quad\wedge\quad\bigl(\forall c\in\Cstar\quad h(c)\geq0\bigr).
\end{gather}

In this case, the vector $\bigl(h(c)\bigr)_{c\in\Cstar}$ is a probability vector. Indeed, it is non negative and it sums up to $1$ thanks to~\eqref{eq:3}, since $f(\vd)=1$ and $h(\vd)=0$.

\paragraph*{Markov chain of cliques.}
\label{sec:mark-chains-cliq}

\cite{abbes15} If $f$ is a probabilistic valuation, then there exists a unique probability measure $\nu$ on $\BM$ equipped with the natural Borel \slgb, such that $\nu(\up x)=f(x)$ for all $x\in\M$, where $\up x$ is the \emph{visual cylinder} defined by $\up x=\{\omega\in\BM\tq x\leq\omega\}$.

With respect to this probability measure, the sequence of mappings $C_i:\BM\to\Cstar$, seen as a sequence of random variables, is a homogeneous Markov chain. Its initial distribution is given by: $\forall c\in\Cstar\quad \nu(C_1=c)=h(c)$, where $h$ is the M\"obius transform of~$f$. The transition matrix of the chain can also be described, but we shall not need it in the sequel.

\paragraph*{Example.}
\label{sec:example}

Let $\M=\langle a,b,c,d\,|\, ad=da,\; bd=db\rangle$. The set of cliques is $\C=\{\vd,\;a,b,c,d,\;ad,bd\}$. Let us simply denote by $a$, $b$, etc, the values of $f(a)$, $f(b)$, etc, for some valuation~$f$. The normalization conditions~(\ref{eq:7}) for $f$ to be a probabilistic valuation are:
\begin{gather*}
  1-a-b-c-d+ad+bd=0\\
  \begin{aligned}
    a-ad&\geq0,&b-bd&\geq0,&c&\geq0,&d&\geq0,&ad&\geq0,&bd&\geq0.
  \end{aligned}
\end{gather*}
A solution is to put $a=b=1/3$ and $c=d=1/4$. Another solution is to put $a=b=c=d=1-\sqrt2/2$. The later value is the root of smallest modulus of the polynomial $1-4p+2p^2$, which we encounter below as the M\"obius polynomial of the monoid.

\paragraph*{Growth series and M\"obius polynomials.}
\label{sec:growth-series-mobius}

The  \emph{growth series} $G(z)$ and the \emph{M\"obius polynomial} $\mu(z)$ of $\M$ are defined as follows:
\begin{align*}
  G(z)&=\sum_{x\in\M}z^{|x|},&\mu(z)&=\sum_{c\in\C}(-1)^{|c|}z^{|c|}.
\end{align*}

\cite{cartier69} The series $G(z)$ is rational, and it is the formal inverse of the M\"obius polynomial:\quad $G(z)\mu(z)=1$.

\cite{krob03,goldwurm00} If $\Sigma\neq\emptyset$, the M\"obius polynomial has a unique root of smallest modulus. This root, say~$r$, is real and lies in $(0,1]$. If $\Sigma=\emptyset$, we put $r=\infty$. In all cases, the radius of convergence of $G(z)$ is~$r$.

We note that: \emph{$r\geq1$ if and only if $\M$ is commutative}---an elementary result to be generalised when dealing with deterministic concurrent systems in Sections~\ref{sec:determ-conc-syst} and~\ref{sec:irred-determ-conc}. Indeed, if $\M$ is not commutative, then $\M$ contains the free monoid on two generators as a submonoid, hence $r\leq1/2$. Whereas, if $\M$ is commutative and $\Sigma$ has $N\geq0$ elements, then $\mu(z)=(1-z)^N$ and therefore $r=1$ or $r=\infty$. In this case, one recovers  from the formula $G(z)=1/(1-z)^N$ the standard elementary result that commutative free monoids have a polynomial growth.

% If $\Sigma\neq\emptyset$ and  $f$ is the valuation satisfying $f(a)=r$ for all $a\in\Sigma$, then $f$ is a probabilistic valuation, and $r$ is the unique value with this property---other probabilistic valuations exists in general, but not necessarily oncstant on~$\Sigma$. The associated probability measure on $\BM$ is the \emph{uniform measure at infinity for~$\M$}.

\paragraph*{Representation of traces.}
\label{sec:repr-trac}

The alphabet $\Sigma$ is usually represented by its \emph{Coxeter graph}~\cite{dehornoy15}, which is the graph $(\Sigma,D)$ with all self-loops omitted. Hence two distinct letters commute with each other if and only if they are not joined by an edge; see an example depicted on Fig.~\ref{fig:coacosn}.

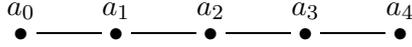
\begin{figure}[t]
  \centering
  $$
\xymatrix{\bullet\ar@{-}[r]\labelu{a_0}&\bullet\ar@{-}[r]\labelu{a_1}&\bullet\ar@{-}[r]\labelu{a_2}&\bullet\ar@{-}[r]\labelu{a_3}&\bullet\labelu{a_4}}
  $$
  \caption{Coxeter graph of the trace monoid $\M(\Sigma,I)$ with $\Sigma=\{a_0,\ldots,a_4\}$ and $(a_i,a_j)\in I\iff |i-j|\geq2$. The set of cliques is $\C=\{\vd,\;\mbox{$a_0,\ldots,a_4$},\;\mbox{$a_0\cdot a_2$},a_0\cdot a_3,\;a_0\cdot a_4,\;a_1\cdot a_3,\;a_1\cdot a_4,\;a_2\cdot a_4,\quad a_0\cdot a_2\cdot a_4\}$.}
  \label{fig:coacosn}
\end{figure}

A convenient representation of traces is provided by the identification of traces with the \emph{heaps of pieces} introduced in~\cite{viennot86}. Picture each letter as a piece falling to the ground, in such a way that distinct letters which commute with each other fall along parallel lines; whereas non commutative letters fall in such a way that they block each other. The heaps of pieces thus obtained are combinatorial object corresponding bijectively to the elements of the trace monoid, by reading the letters labelling the pieces from bottom to top. The cliques of the normal form of a trace correspond to the horizontal layers that appear in the heap of pieces. See an illustration on Fig.~\ref{fig:apoaspokzx}.

\begin{figure}[h]
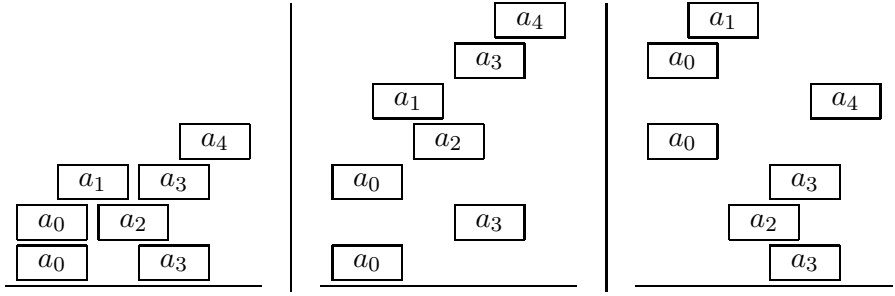

  \centering
    \begin{tabular}{c@{\quad}|@{\quad}c@{\quad}|@{\quad}c}
\xy
<.1em,0em>:
0="G",
"G"+(12,6)*{a_0},
"G";"G"+(24,0)**@{-};"G"+(24,12)**@{-};"G"+(0,12)**@{-};"G"**@{-},
(42,0)="G",
         "G"+(12,6)*{a_3},
"G";"G"+(24,0)**@{-};"G"+(24,12)**@{-};"G"+(0,12)**@{-};"G"**@{-},
(0,14)="G";         
         "G"+(12,6)*{a_0},
"G";"G"+(24,0)**@{-};"G"+(24,12)**@{-};"G"+(0,12)**@{-};"G"**@{-},
(28,14)="G";         
         "G"+(12,6)*{a_2},
"G";"G"+(24,0)**@{-};"G"+(24,12)**@{-};"G"+(0,12)**@{-};"G"**@{-},
(14,28)="G";         
         "G"+(12,6)*{a_1},
"G";"G"+(24,0)**@{-};"G"+(24,12)**@{-};"G"+(0,12)**@{-};"G"**@{-},
(42,28)="G";         
         "G"+(12,6)*{a_3},
"G";"G"+(24,0)**@{-};"G"+(24,12)**@{-};"G"+(0,12)**@{-};"G"**@{-},
(56,42)="G";         
         "G"+(12,6)*{a_4},
"G";"G"+(24,0)**@{-};"G"+(24,12)**@{-};"G"+(0,12)**@{-};"G"**@{-},
(-4,-2);(84,-2)**@{-}
\endxy
&
 \xy
<.1em,0em>:
0="G",
"G"+(12,6)*{a_0},
"G";"G"+(24,0)**@{-};"G"+(24,12)**@{-};"G"+(0,12)**@{-};"G"**@{-},
(42,14)="G",
         "G"+(12,6)*{a_3},
"G";"G"+(24,0)**@{-};"G"+(24,12)**@{-};"G"+(0,12)**@{-};"G"**@{-},
(0,28)="G";         
         "G"+(12,6)*{a_0},
"G";"G"+(24,0)**@{-};"G"+(24,12)**@{-};"G"+(0,12)**@{-};"G"**@{-},
(28,42)="G";         
         "G"+(12,6)*{a_2},
"G";"G"+(24,0)**@{-};"G"+(24,12)**@{-};"G"+(0,12)**@{-};"G"**@{-},
(14,56)="G";         
         "G"+(12,6)*{a_1},
"G";"G"+(24,0)**@{-};"G"+(24,12)**@{-};"G"+(0,12)**@{-};"G"**@{-},
(42,70)="G";         
         "G"+(12,6)*{a_3},
"G";"G"+(24,0)**@{-};"G"+(24,12)**@{-};"G"+(0,12)**@{-};"G"**@{-},
(56,84)="G";         
         "G"+(12,6)*{a_4},
"G";"G"+(24,0)**@{-};"G"+(24,12)**@{-};"G"+(0,12)**@{-};"G"**@{-},
(-4,-2);(84,-2)**@{-}
\endxy
&
 \xy
<.1em,0em>:
(0,70)="G",
"G"+(12,6)*{a_0},
"G";"G"+(24,0)**@{-};"G"+(24,12)**@{-};"G"+(0,12)**@{-};"G"**@{-},
(42,0)="G",
         "G"+(12,6)*{a_3},
"G";"G"+(24,0)**@{-};"G"+(24,12)**@{-};"G"+(0,12)**@{-};"G"**@{-},
(0,42)="G";         
         "G"+(12,6)*{a_0},
"G";"G"+(24,0)**@{-};"G"+(24,12)**@{-};"G"+(0,12)**@{-};"G"**@{-},
(28,14)="G";         
         "G"+(12,6)*{a_2},
"G";"G"+(24,0)**@{-};"G"+(24,12)**@{-};"G"+(0,12)**@{-};"G"**@{-},
(14,84)="G";         
         "G"+(12,6)*{a_1},
"G";"G"+(24,0)**@{-};"G"+(24,12)**@{-};"G"+(0,12)**@{-};"G"**@{-},
(42,28)="G";         
         "G"+(12,6)*{a_3},
"G";"G"+(24,0)**@{-};"G"+(24,12)**@{-};"G"+(0,12)**@{-};"G"**@{-},
(56,56)="G";         
         "G"+(12,6)*{a_4},
"G";"G"+(24,0)**@{-};"G"+(24,12)**@{-};"G"+(0,12)**@{-};"G"**@{-},
(-4,-2);(84,-2)**@{-}
\endxy
    \end{tabular}
    \caption{In this example the commutation relations are those of the Coxeter graph depicted on Fig.~\ref{fig:coacosn}. \textsl{Left:} representation as a heap of piece of the trace which normal form is $(a_0a_3,a_0a_2,a_1a_3,a_4)$.  \textsl{Middle and right:} representations of two words in the congruence class of the trace~$x$: $a_0$-$a_3$-$a_0$-$a_2$-$a_1$-$a_3$-$a_4$ (middle) and $a_3$-$a_2$-$a_3$-$a_0$-$a_4$-$a_0$-$a_1$ (right).}
  \label{fig:apoaspokzx}
\end{figure}

\subsection{Concurrent systems and their combinatorics}
\label{sec:conc-syst-their}

The background material presented in this section is borrowed from~\cite{abbes19:_markov,abbes20}.

\paragraph*{Concurrent systems and executions.}
\label{sec:basics}

A \emph{concurrent system} is a triple $(\M,X,\bot)$ where $\M$ is a trace monoid, $X$~is a finite set of \emph{states} and $\bot$~is a special symbol not in~$X$, together with a right monoid action of $\M$ on $X\cup\{\bot\}$, denoted by $(\alpha,x)\mapsto\alpha\cdot x$, and such that $\bot\cdot x=\bot$ for all $x\in\M$. By definition of a monoid action, one has thus $\alpha\cdot(x\cdot y)=(\alpha\cdot x)\cdot y$ for all $(\alpha,x,y)\in X\times\M\times \M$, and $\alpha\cdot\vd=\alpha$ for all $\alpha\in X$.

The concurrent system $\X$ is \emph{trivial} if $\alpha\cdot a=\bot$ for all $\alpha\in X$ and for all $a\in\Sigma$. It is \emph{non trivial} otherwise.

The symbol $\bot$ represents a sink state. So we are interested, for every $\alpha,\beta\in X$, in the following subsets of~$\M$:
\begin{align*}
  \M_{\alpha,\beta}&=\{x\in\M\tq \alpha\cdot x=\beta\},&\M_\alpha&=\{x\in\M\tq \alpha\cdot x\neq\bot\}.
\end{align*}

Traces of $\M_\alpha$ are called \emph{executions starting from~$\alpha$}, or \emph{executions} for short if the context is clear. Note that $\M_\alpha$ is always downward closed in $(\M,\leq)$.

\medskip
We introduce the following useful notations, for $\alpha,\beta\in X$:
\begin{align*} \Sigma_\alpha&=\Sigma\cap\M_\alpha&\C_\alpha&=\C\cap\M_\alpha&
\Cstar_\alpha&=\Cstar\cap\M_\alpha&                                                                              \C_{\alpha,\beta}&=\C\cap\M_{\alpha,\beta}
\end{align*}

A \emph{generalised execution from~$\alpha$} is an element $\xi\in\Mbar$ such that:
\begin{gather*}
  \forall x\in \M\quad x\leq\xi\implies x\in\M_\alpha.
\end{gather*}
Their set is denoted~$\Mbar_\alpha$, and we also put $\BM_\alpha=\Mbar_\alpha\cap\BM$.

As a running example for a ``general concurrent system'', we use the $1$-safe Petri net depicted in Fig.~\ref{fig:petrinetrs},~$(a)$. The underlying trace monoid is generated by the transitions, with commutative transitions $t$ and $t'$ whenever $^{\bullet}t^\bullet\cap{}^\bullet{t'}^\bullet=\emptyset$, thus $\M=\langle a,b,c,d\;|\; ad=da,\ db=db\rangle$. The corresponding Coxeter graph is depicted on Fig.~\ref{fig:petrinetrs},~$(b)$, and the graph of marking is depicted on Fig.~\ref{fig:petrinetrs},~$(c)$.

\begin{figure}
  \centering
  \begin{tabular}{c|c|c}
\begin{minipage}[c]{.6\textwidth}
  $$
  \xymatrix@C=.5em@M=.2em{
    &\place\markp\labelr{A}\POS[]\ar[dl]!U\ar[dr]!U\\
    \transh\ar@(u,l)[ur]\labell{a}&&\transh\ar[d]\labelr{b}\\
&    &\place\ar[dr]!U\labelr{B}&&\place\ar[dl]!U\ar[dr]!U!U\POS[]\markp\labell{C}\\
 &   &&\transh\POS!R\ar@(r,u)[ur]!D\POS[]\POS!L\ar@(l,d)[lluuu]!D\labeld{\strut c}&&\transh\ar@(u,r)[ul]\labeld{d}
    }
    $$
  \end{minipage}
     &
       \begin{minipage}[c]{.2\textwidth}
         $$
         \xymatrix@C=1em{
           \bullet\ar@{-}[r]\ar@{-}[d]\labelu{a}&\bullet\ar@{-}[dl]\labelu{b}\\
           \bullet\ar@{-}[d]\labelr{c}\\
           \bullet\labelr{d}
}
$$
\end{minipage}
&           \begin{minipage}[c]{.2\textwidth}
    $$
\xymatrix{*++[o][F-]{\alpha_0}%
\POS!L\ar@(l,ul)^{d}[]!U!L(.2)%
\POS!R\ar@(r,ur)_{a}[]!U!R(.2)%
\POS[]\ar@<1ex>^{b}[d]\ar@<1ex>^{c}[d];[]\\
  *++[o][F-]{\alpha_1}%\ar@(ul,dl)_{d}
\POS!L\ar@(l,dl)_{d}[]!D
    }
    $$
  \end{minipage}
    \\
    $(a)$&$(b)$&$(c)$\\
    \hline
\multicolumn{3}{c}{%
}\\
 \multicolumn{3}{c}{%
       \begin{minipage}[c]{\textwidth}
     $$
     \xymatrix@C=1em{
     & &&*+[F]{(\alpha_0,d)}\POS!R!D(.5)\ar@(dr,ur)[]!R!U(0.5)&&
\\     &*+[F]{(\alpha_0,ad)}\POS!U\ar[urr]!L\POS[]\ar[rrr]\POS!L\ar[dl]!U\POS[]\ar[d]\POS[]\POS!L!D(0.5)\ar@(dl,dr)[]!D!L(0.5)
      &&&*+[F]{(\alpha_0,bd)}\POS!D\ar[dl]!R\ar[ddl]!R&\\
      *+[F]{(\alpha_0,a)}\ar[r]\POS!L!D(0.5)\ar@(dl,dr)[]!D!L(0.5)
      &*+[F]{(\alpha_0,b)}
      &&*+[F]{(\alpha_1,c)}\POS!U!R(0.25)\ar[ur]!D!L\POS[]\ar@<1ex>[ll]\ar@<1ex>[ll];[]\POS!L!U\ar[ull]!R!D
      \POS[]\ar'[u][uu]\POS!D!L(0.5)\ar@(dl,dr)[lll]!D!R(0.5)
\\
&&&*+[F]{(\alpha_1,d)}\POS!L!D(0.5)\ar@(dl,dr)[]!D!L(0.5)\POS[]\ar[u]&
       }
       $$
     \end{minipage}
                                                                       }
     \\
 \multicolumn{3}{c}{%
 }    \\[1em]
 \multicolumn{3}{c}{%
     $(d)$
     }
  \end{tabular}                                                                              
  \caption{\small $(a)$---A safe Petri net with its initial marking $\alpha_0=\{A,C\}$ depicted. The two reachable markings are  $\alpha_0$ and $\alpha_1=\{B,C\}$.\quad$(b)$---The Coxeter graph of the associated trace monoid.\quad $(c)$---Graph of markings of the net.\quad $(d)$---Digraph of states-and-cliques of the associated concurrent system.}
  \label{fig:petrinetrs}
\end{figure}
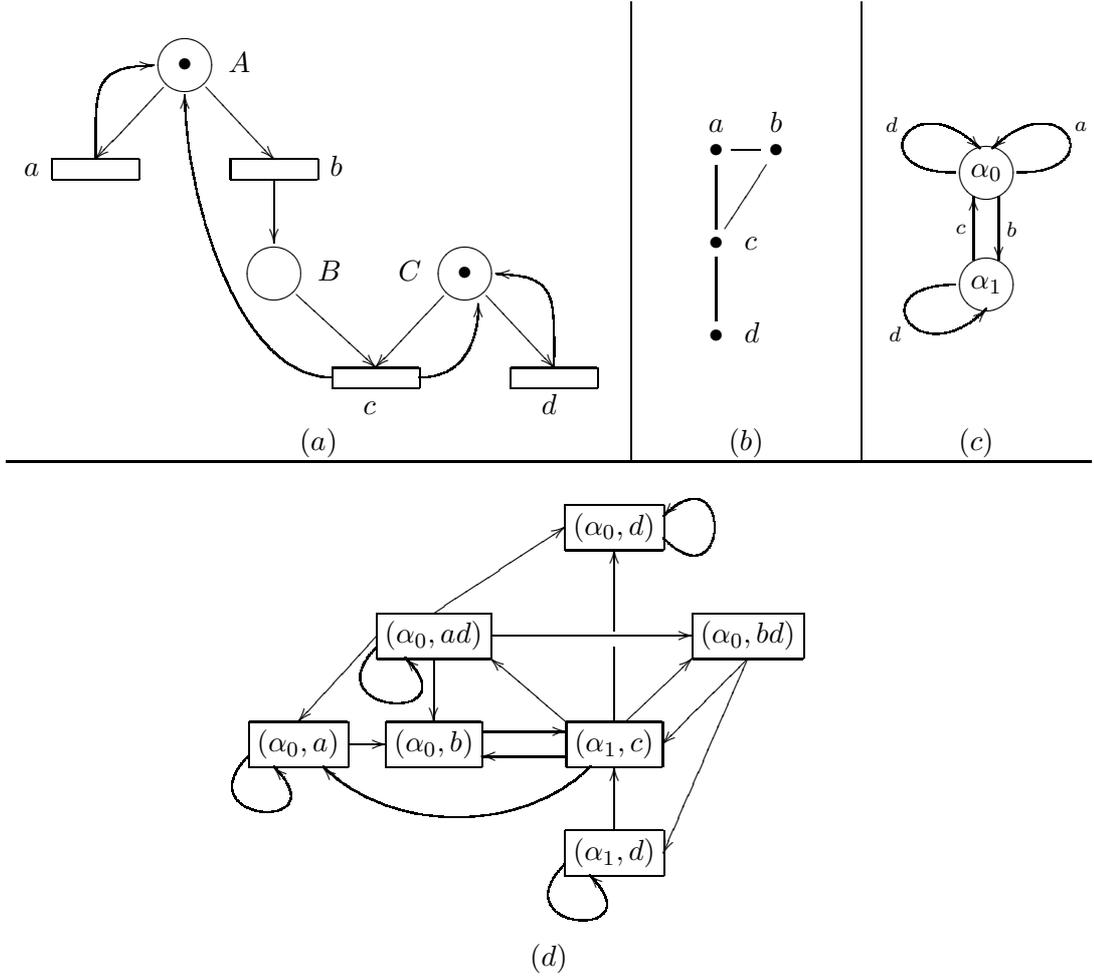

\paragraph*{Digraph of states-and-cliques.}
\label{sec:digr-stat-cliq}

Generalised executions of a concurrent system $\X=(\M,X,\bot)$ are generalised traces of~$\M$. As seen in Sect.~\ref{sec:generalised-traces}, generalised traces correspond to paths in the digraph of cliques $(\C,\to)$. Not all paths of $(\C,\to)$ however correspond, in general, to executions of~$\X$. In order to take into account the constraints induced by the monoid action, we introduce the \emph{digraph of states-and-cliques} $(\D,\to)$, the vertices of which are pairs $(\alpha,c)$ with $\alpha$ ranging over $X$ and $c$ ranging over~$\C_\alpha$. There is an arrow $(\alpha,c)\to(\beta,d)$ in $\D$ if $\beta=\alpha\cdot c$ and if $(c,d)$ is a normal pair of cliques.

To every generalised execution $\xi=(c_i)_{i\geq1}$ with $\xi\in\Mbar_\alpha$, is associated the path $(\alpha_{i-1},c_i)_{i\geq1}$ in~$\D$, where $\alpha_i$ is defined by $\alpha_0=\alpha$ and $\alpha_i=\alpha\cdot(c_1\cdot\ldots\cdot c_i)$ for $i\geq1$. We put $Y_i(\xi)=(\alpha_{i-1},c_i)$ for every integer $i\geq1$.

Conversely, every infinite path in $\D$ corresponds to a unique generalised execution. Consider the subgraph $\Dstar$ of $\D$ with all vertices of the form $(\alpha,c)$ with $c\neq\vd$. Then infinite paths in $\Dstar$ correspond bijectively to infinite executions.

For our running example, the digraph of states-and-cliques is depicted on Fig.~\ref{fig:petrinetrs},~$(d)$.

\paragraph*{Characteristic root.}
\label{sec:characteristic-root}

The combinatorics of a concurrent system $\X=(\M,X,\bot)$ involves not only the combinatorics of~$\M$, but also of the monoid action $X\times\M\to X$. Consider the \emph{M\"obius matrix} $\mu(z)=(\mu_{\alpha,\beta}(z))_{(\alpha,\beta)\in X\times X}$, the polynomial~$\theta(z)$, and the growth matrix $G(z)=(G_{\alpha,\beta}(z))_{(\alpha,\beta)\in X\times X}$ defined by:
\begin{align*}
  \mu_{\alpha,\beta}(z)&=\sum_{c\in\C_{\alpha,\beta}}(-1)^{|c|}z^{|c|}&\theta(z)&=\det\mu(z)&G_{\alpha,\beta}(z)=\sum_{x\in\M_{\alpha,\beta}}z^{|x|}
\end{align*}

Then $G(z)$ is a matrix of rational series, and it is the inverse of the M\"obius matrix: $G(z)\mu(z)=\text{Id}$. One of the roots of smallest modulus of the polynomial $\theta(z)$ is real and lies in $(0,1]\cup\{\infty\}$, with the convention that it is $\infty$ if $\theta(z)$ is a non zero constant. This non negative real or $\infty$ is the \emph{characteristic root} of the concurrent system~$\X$. The characteristic root $r$ is the minimum of all convergence radii of the generating series $G_{\alpha,\beta}(z)$, for $(\alpha,\beta)$ ranging over~$X\times X$.

For our running example, the M\"obius matrix is given by: \begin{gather*}
  \mu(z)=
  \begin{array}{c}
    \alpha_0\\\alpha_1
  \end{array}
  \begin{pmatrix}
    1-2z+z^2&-z+z^2\\
    -z&1-z
  \end{pmatrix}
\end{gather*}
with determinant $\theta(z)=(1-z)^2(1-2z)$. The characteristic root is thus $r=1/2$.

\paragraph*{Irreducibility and the spectral property.}
\label{sec:irreducibility}

A concurrent system $\X=(\M,X,\bot)$ is \emph{irreducible} if:
\begin{inparaenum}[1)]
\item The monoid $\M$ is irreducible;
\item $\M_{\alpha,\beta}\neq\emptyset$ for all $\alpha,\beta\in X$;
\item For every $\alpha\in X$ and for every letter $a\in \Sigma$ there exists $x\in\M_\alpha$ such that $a\in x$.
\end{inparaenum}

If $\Sigma'$ is any subset of~$\Sigma$, and if $\M'=\langle\Sigma'\rangle$ is the submonoid of $\M$ generated by~$\Sigma'$, then the restriction of the action $(X\cup\{\bot\})\times\M'\to X\cup\{\bot\}$ defines clearly a new concurrent system $\X'=(\M',X,\bot)$, said to be \emph{induced by restriction}. In particular, let $\X^a$ denote the concurrent system induced by restriction with $\Sigma'=\Sigma\setminus\{a\}$, and let $r^a$ be the characteristic root of~$\X^a$.

A key property, that we shall use later, is the \emph{spectral property}~\cite{abbes20} which states:\quad \emph{if $\X$ is irreducible, then $r^a>r$ for every  $a\in\Sigma$}.

The concurrent system in our running example from Fig.~\ref{fig:petrinetrs} is irreducible.

\paragraph*{Valuations and probabilistic valuations. Markov chain of states-and-cliques.}
\label{sec:valu-prob-valu}

A \emph{valuation} on a concurrent system $\X=(\M,X,\bot)$ is a family $f=(f_\alpha)_{\alpha\in X}$ of mappings $f_\alpha:\M\to\bbR_{\geq0}$ satisfying the three following properties:
\begin{gather}
  \label{eq:4}
  \forall \alpha\in X\quad\forall x\in\M\quad \alpha\cdot x=\bot\implies f_\alpha(x)=0\\
  \label{eq:5}
  \forall\alpha\in X\quad\forall x\in\M_\alpha\quad\forall y\in\M_{\alpha\cdot x}\quad f_{\alpha}(x\cdot y)=f_\alpha(x)f_{\alpha\cdot x}(y)\\
  \label{eq:6}
  \forall\alpha\in X\quad f_\alpha(\vd)=1
\end{gather}

Let $f=(f_\alpha)_{\alpha\in X}$ be a valuation and for each $\alpha\in X$, let $h_\alpha:\C\to\bbR$ be the M\"obius transform of the restriction $f_\alpha\rest\C:\C\to\bbR_{\geq0}$. Note first that $h_\alpha(x)=0$ if $x\notin\M_\alpha$. We say that $f$ is a \emph{probabilistic valuation} if:
\begin{gather}
  \label{eq:8}
  \forall\alpha\in X\quad\bigl(h_\alpha(\vd)=0\quad\land\quad(\forall c\in\Cstar_\alpha\quad h_\alpha(c)\geq0)\bigr).
\end{gather}

In this case, there exists a unique family $\nu=(\nu_\alpha)_{\alpha\in X}$, where $\nu_\alpha$ is a probability measure on~$\BM_\alpha$, such that $\nu_\alpha(\up x)=f_\alpha(x)$ for all $\alpha\in X$ and for all $x\in\M_\alpha$. Of course the existence of a probabilistic valuation implies in particular that $\BM_\alpha\neq\emptyset$, a property which might not be satisfied in general even if $\Sigma\neq\emptyset$.

If $\nu=(\nu_\alpha)_{\alpha\in X}$ is associated as above with a probabilistic valuation $f=(f_\alpha)_{\alpha\in X}$, then for each state $\alpha\in X$, and with respect to the probability measure~$\nu_\alpha$\,, the family of mappings $Y_i:\BM_\alpha\to\Dstar$ defined earlier is a homogeneous Markov chain, called the \emph{Markov chain of states-and-cliques}. Its initial distribution is given by $\un_{\alpha}\otimes h_\alpha$; hence in particular:
\begin{gather}
  \label{eq:9}
\forall\alpha\in X\quad\forall c\in\Cstar_\alpha\quad  \nu_\alpha(C_1=c)=h_\alpha(c).
\end{gather}

Let us determine all the probabilistic valuations for the running example of Fig.~\ref{fig:petrinetrs}. Any probabilistic valuation $f=(f_\alpha)_{\alpha\in X}$ is entirely determined by the \emph{finite} family of values $f_\alpha(u)$ for $(\alpha,u)$ ranging over $\{\alpha_0,\alpha_1\}\times \Sigma$, since then the other values $f_\alpha (x)$ are obtained by the chain rule $f_\alpha(xy)=f_\alpha(x)f_{\alpha\cdot x}(y)$. 

Since $f_{\alpha_0}(c)=f_{\alpha_1}(a)=f_{\alpha_2}(b)=0$, the remaining parameters for $f$ are $p=f_{\alpha_0}(a)$, $q=f_{\alpha_0}(b)$, $s=f_{\alpha_0}(d)$, $t=f_{\alpha_1}(c)$, $u=f_{\alpha_1}(d)$. The parameters are not independent; to cope with the commutativity relations induced by the trace monoid, one must have $f_{\alpha_0}(a)f_{\alpha_0\cdot a}(d)=f_{\alpha_0}(d)f_{\alpha_0\cdot d}(a)$,  since $ad=da$, and $f_{\alpha_0}(b)f_{\alpha_1}(d)=f_{\alpha_0}(d)f_{\alpha_0\cdot d}(b)$ since $bd=db$; yielding simply $u=s$ here.

\begin{table}[t]
  \centering
$  \begin{array}{c|c|c|c|c|c|c|c}
      \text{state $\alpha$}    &h_\alpha(\varepsilon)&h_\alpha(a)&h_\alpha(b)&h_\alpha(c)&h_\alpha(d)&h_\alpha(ad)&h_\alpha(bd)\\
      \hline
      \alpha_0&1-p-q-s+ps+qs
                &p-ps&q-qs&0&s-ps-qs&ps&qs\\
      \alpha_1&
                1-t-s
                                                     &0&0&t&s&0&0\\
    \end{array}
    $
  \caption{M\"obius tranform of a generic valuation for the running example depicted in Fig.~\ref{fig:petrinetrs}, with parameters $p=f_{\alpha_0}(a)$, $q=f_{\alpha_0}(b)$, $s=f_{\alpha_0}(d)=f_{\alpha_1}(d)$ and $t=f_{\alpha_1}(c)$.}
  \label{tab:moniuadsa}
\end{table}

The M\"obius tranform of $f_{\alpha_0}$ evaluated for instance at $b$ is $h_{\alpha_0}(b)=f_{\alpha_0}(b)-f_{\alpha_0}(bd)=f_{\alpha_0}(b)-f_{\alpha_0}(b)f_{\alpha_1}(d)=q-qs$. Other computations are done similarly, and we gather the results in Table~\ref{tab:moniuadsa}.
According to~(\ref{eq:8}), the normalization contraints on the parameters for the valuation $f$ to be probabilistic are thus:
\begin{align}
  \label{eq:10}
  1-p-q-s+ps+qs&=0,&
                1-t-s&=0,
\end{align}
plus all inequalities $h_{\alpha_0}(a)\geq0$, etc, which in this case amount to specify that all parameters vary between $0$ and~$1$. The second equality in~(\ref{eq:10}) is standard: since there is no concurrenycy enabled at~$\alpha_1$, the events of firing $c$ and $d$ are disjoint, hence their probabilities sum up to~$1$. The first equality in~(\ref{eq:10}) is less standrad. It takes into account the existence of concurrency enabled at~$\alpha_0$ and shows a degree greater than~$1$, resulting form the existence of cliques of order~$2$.

Here, the equality $h_{\alpha_0}(\varepsilon)=0$ rewrites as $(1-p-q)(1-s)=0$. It follows that, if $s\neq1$, then $1-p-q=0$ and therefore $h_{\alpha_0}(d)=s(1-p-q)=0$. Hence the node $(\alpha_0,d)$ is never reached, which meets well the inuition. We say that $(\alpha_0,d)$ is a \emph{null node}. See \cite{abbes20} for more details about the notion of null node.

\paragraph*{Representation of concurrent systems and of executions.}
\label{sec:repr-conc-syst}

To represent a concurrent system $\X=(\M,X,\bot)$, we first use the Coxeter graph of~$\M$, as in Fig.~\ref{fig:coacosn}. We also depict the \emph{labelled multigraph of states}, which vertices are the elements of~$X$, and with an edge from $\alpha$ to~$\beta$ labelled by the letter $a\in\Sigma$ if $\alpha\cdot a=\beta$, as in Fig.~\ref{fig:petrinetrs},~$(c)$. For representing executions, we stick to the representation by heaps of pieces introduced earlier for traces.

\begin{remark}
  Any multigraph $V$ with edges labelled by elements from a set $\Sigma$ represents an action of the free monoid $(V\cup\{\bot\})\times\Sigma^*\to(V\cup\{\bot\})$, provided that for any node $v\in V$, there is no two edges starting from $v$ and labelled with the same letter. It requires an additional verification to check that it also represents an action of a trace monoid $\M=\M(\Sigma,I)$ on~$V$; namely, one has to check that $\alpha\cdot(ab)=\alpha\cdot(ba)$ for any two letters $(a,b)\in I$.
\end{remark}

\subsection{A comparison result}
\label{sec:an-elem-comp}

In this subsection, we state an elementary lemma and its corollary, both belonging to trace theory, and given in a form slightly more general than precisely needed in the sequel. 

Consider an alphabet~$\Sigma$ and two independence relations $I$ and $I'$ on $\Sigma$ such that $I\subseteq I'$, and consider the two trace monoids $\M=\M(\Sigma,I)$ and $\N=\M(\Sigma,I')$. There is a natural surjection $\pi:\M\to\N$, which entails in particular that $\M$ is ``not smaller'' than~$\N$. It seems to have been unnoticed so far that, when restricted to the set of sub-traces of a given trace of~$\M$, or even of~$\Mbar$, then $\pi$ becomes injective. This is the topic of the following lemma.

The lemma generalises the following elementary fact. Let $\M=\Sigma^*$ be a free monoid and let $u\in\Sigma^*$. Then any prefix word $x\leq u$ is entirely determined by the collection $(n_a)_{a\in\Sigma}$ where $n_a$ is the number of occurrences of the letter $a$ in~$x$. Hence $x$ is entirely determined by its image in the free commutative monoid generated by~$\Sigma$.

\begin{lemma}
  \label{lem:1}
  Let $I\subseteq I'$ be two independence relations on an alphabet\/~$\Sigma$, let $\M=\M(\Sigma,I)$ and $\N=\M(\Sigma,I')$, and let $\pi:\M\to\N$ be the natural surjection. Then $\pi$ extends naturally to a surjection on generalised traces, as a mapping still denoted by $\pi:\Mbar\to\Nbar$. Let $\omega\in\Mbar$, and define:\quad
    $\Mbar_{\leq\omega}=\{x\in\Mbar\tq x\leq\omega\}$. Then the restriction of\/ $\pi$ to $\Mbar_{\leq\omega}$ is injective.
\end{lemma}

\begin{proof}
  The extension of $\pi$ to a mapping $\Mbar\to\Nbar$ follows from the definitions, hence we focus on proving that the restriction of $\pi$ to $\Mbar_{\leq\omega}$ is injective. Let $x\in\Mbar_{\leq\omega}$ and let $y=\pi(x)$. Let $c_1$ be the first clique in the normal form of~$x$, and let $d_1$ be the first clique in the normal form of~$y$. Let also $C_1$ be the first clique in the normal form of~$\omega$. We assume with loss of generality that $x\neq\vd$ since $\pi^{-1} (\{\vd\})=\{\vd\}$.

  We claim that $c_1=d_1\cap C_1$. The inclusion $c_1\subseteq d_1\cap C_1$ is clear since both inclusions $c_1\subseteq d_1$ and $c_1\subseteq C_1$ are obvious. For proving the converse inclusion, seeking a contradiction, we assume that there is a letter $a\in d_1\cap C_1$ such that $a\notin c_1$. Then, since $y=\pi(x)$, the letter $a$ belongs to some higher clique in the normal form of~$x$. But, since $x\leq\omega$, and since $a\in C_1$, that entails that $a\in c_1$, contradicting the assumption $a\notin c_1$. Hence $c_1=d_1\cap C_1$, as claimed.

  Repeating inductively the same reasoning, with $x'=c_1 \backslash x$ and with $y'=\pi(x')=c_1\backslash y$ and $\omega'=c_1\backslash \omega$ in place of $x$ and of $y$ and of $\omega$ respectively\footnote{Recall that, if $c\leq u$ with $c,u\in\M$, we denote by $c\backslash u$ the left cancellation of $u$ by~$c$, which is the unique trace $v\in\M$ such that $c\cdot v=u$.}, we see that all the cliques $(c_i)_{i\geq1}$ of the generalised trace $x$ can be reconstructed from~$y$. This entails that $\pi$ is injective.
\end{proof}

\begin{corollary}
  \label{cor:1}
  Let $\M$ be a trace monoid, and let $\omega\in\BM$ be an infinite trace. For each integer $n\geq0$, consider:
  \begin{align*}
    \M_{\leq\omega}(n)&=\{x\in\M\tq x\leq\omega\land |x|=n\},& p_n&=\#\M_{\leq\omega}(n).
  \end{align*}
Then there is a polynomial $P\in\bbZ[X]$ such that $p_n\leq P(n)$ for all integers~$n$. Furthermore, the set $\BM_{\leq\omega}=\{\xi\in\BM\tq \xi\leq\omega\}$ is at most countable. The polynomial $P$ only depends on~$\M$, and not on~$\omega$.
\end{corollary}

\begin{proof}
  Let $\M=\M(\Sigma,I)$ and let $\N$ be the free commutative monoid generated by~$\Sigma$, \ie, $\N=\M(\Sigma,I')$ with $I'=(\Sigma\times\Sigma)\setminus\Delta$ and $\Delta=\{(x,x)\; :\ x\in\Sigma\}$.

  For each integer~$n$, let $q_n=\#\N(n)$. Then it is well known that $q_n=P(n)$ for some polynomial $P\in\bbZ[X]$ (a short proof based on the M\"obius inversion formula was given in Sect.~\ref{sec:growth-series-mobius}). Since $I\subseteq I'$, it follows  from Lemma~\ref{lem:1} that $p(n)\leq q(n)$.

  Furthermore, $\Nbar$~itself is at most countable since $\Nbar$ identifies with:
  \begin{gather*}
    \Nbar\sim\bigl\{(x_i)_{i\in\Sigma}\tq x_i\in\bbZ_{\geq0}\cup\{\infty\},\quad\exists i\in\Sigma\quad x_i=\infty
    \bigr\}.
  \end{gather*}
Hence, the fact that $\BM_{\leq\omega}$ is at most countable also follows from Lemma~\ref{lem:1}.
\end{proof}

\begin{remark}
  Of course, the direct argument:
  \begin{gather*}
\BM_{\leq\omega}\subseteq\bigl\{\xi\in\Cstar^{\bbZ_{\geq1}}\tq \forall i\geq1\quad C_i(\xi)\subseteq C_i(\omega)\bigr\}
\end{gather*}
would not allow to conclude as in Corollary~\ref{cor:1} that $\BM_{\leq\omega}$ is at most countable.
\end{remark}

% \begin{remark}
% The results of Corollary~\ref{cor:1} are not true for monoids very much closed to trace monoids, such as positive braid monoids. Indeed, let $\B$ be a braid monoid on $n\geq3$ strands. The Garside normal form of braids is the analogous of the Cartier-Foata normal form for traces, and the notion of infinite braid is developed in the analogous way, see for instance~\cite{abbes17:_unifor,dehornoy15}. Let $\Delta$ denote the Garside element of~$\B$, and let $\Delta_\infty$ be infinite braid which generalised normal form is $(\Delta,\Delta,\ldots)$. Then it is true that every braid $x\in\B$ satisfies $x\leq\Delta_\infty$\,. Yet, both conclusions of Corollary~\ref{cor:1} are false: the growth of $\B$ is exponential, and its set of infinite braids in uncountable.
% \end{remark}

\section{Deterministic concurrent systems}
\label{sec:determ-conc-syst}

\begin{definition}
  A \emph{deterministic concurrent system (\DCS)} is a concurrent system $\X=(\M,X,\bot)$ such that for every state $\alpha\in X$, the partial order $(\M_\alpha,\leq)$ is a lattice.
\end{definition}

\begin{remark}
\label{rem:1} According to the background on \lub\ and \glb\ on trace monoids recalled in Section~\ref{sec:lower-upper-bounds} on the one hand, and since $\M_\alpha$ is a downward closed subset of~$\M$ on the other hand, we have for any two executions $x,y\in\M_\alpha$:
  \begin{inparaenum}[1)]
  \item $x$~and $y$ have a \glb\ in~$\M_\alpha$, which coincides with their \glb\ in~$\M$; and
  \item $x$~and $y$ have a \lub\ in $\M_\alpha$ if and only they have a common upper bound in~$\M_\alpha$, in which case their \lub\ in $\M_\alpha$ coincides with their \lub\ in~$\M$.
  \end{inparaenum}
Note however that the existence of $x\vee y$ in $\M$ is not enough to insure that $x\vee y\in\M_\alpha$.

  Henceforth, a concurrent system $(\M,X,\bot)$ is a \DCS\ if and only if, for every state~$\alpha$, any two executions $x,y\in\M_\alpha$ have a common upper bound in~$\M_\alpha$.
\end{remark}

The following result says that \DCS\ correspond to ``locally commutative'' concurrent systems.

\begin{proposition}
  \label{prop:1}
  Let $\X=(\M,X,\bot)$ be a concurrent system. Then the following properties are equivalent:
  \begin{compactenum}[\normalfont(i)]
  \item\label{item:15} $\X$ is deterministic.
  \item\label{item:16} For every $\alpha\in X$, the partial order $(\C_\alpha,\leq)$ is a lattice.
  \item\label{item:17} For every $\alpha\in X$, any two letters in $\Sigma_\alpha$ commute with each other.
  \end{compactenum}
\end{proposition}

\begin{proof}
  The equivalence $\text{(\ref{item:16})}\iff\text{(\ref{item:17})}$ and the implication $\text{(\ref{item:15})}\implies\text{(\ref{item:17})}$ are clear. The interesting point is the implication $\text{(\ref{item:16})}\implies\text{(\ref{item:15})}$.

  Assume that $(\C_\alpha,\leq)$ is a lattice for every $\alpha\in X$. Fix $\alpha\in X$ and let $x,y\in\M_\alpha$. Assume first that $x\wedge y=\vd$. Let $(c_1,\ldots,c_k)$ and $(d_1,\ldots,d_m)$ be the normal forms of $x$ and of~$y$. Maybe by adding the empty trace at the tail of one or the other normal form, we assume that $k=m$, at the cost of tolerating that some of the elements may be the empty trace.

  On the one hand, since $c_1\cdot c_2$ is an execution starting from~$\alpha$, one has $c_2\in\C_{\alpha\cdot c_1}$. On the other hand, both $c_1$ and $d_1$ belong to~$\C_\alpha$, which is a lattice by assumption. Hence $c_1\vee d_1\in\C_\alpha$. And since $c_1\wedge d_1=\vd$ by assumption, one has $c_1\vee d_1=c_1\cdot d_1=d_1\cdot c_1$. Therefore: $d_1\in\C_{\alpha\cdot c_1}$. Since both cliques $c_2$ and $d_1$ belong to~$\C_{\alpha\cdot c_1}$, which is a lattice, it follows that $c_2\vee d_1\in\C_{\alpha\cdot c_1}$.

  Now we claim that $c_2\wedge d_1=\vd$. Otherwise, there exists a letter $a$ occurring in both $c_2$ and~$d_1$. Since $(c_1,c_2)$ is a normal pair of cliques, there exists $b\in c_1$ such that $(a,b)\in D$, the dependence pair of the monoid. Because of the assumption $c_1\wedge d_1=\vd$, the identity $a=b$ is impossible. But both $a$ and $b$ belong to~$\Sigma_\alpha$, and since $a\neq b$, the fact that $(a,b)\in D$ contradicts that $\C_\alpha$ is a lattice; our claim is proved.

  We have obtained that $c_2\vee d_1$ exists in $\C_{\alpha\cdot c_1}$ and that $c_2\wedge d_1=\vd$. Hence $c_2\vee d_1=c_2\cdot d_1=d_1\cdot c_2$. It implies that $c_2\in\C_{\alpha\cdot(c_1\vee d_1)}$. Symmetrically, we obtain that $d_2\in\C_{\alpha\cdot (c_1\vee d_1)}$. Since $\C_{\alpha\cdot(c_1\vee d_1)}$ is a lattice, it follows that $d_2\vee c_2\in\C_{\alpha\cdot(c_1\vee d_1)}$. But again,  $d_2\wedge c_2=\vd$ hence $d_2\vee c_2=d_2\cdot c_2=c_2\cdot d_2$. Therefore we obtain that the following trace belongs to~$\M_\alpha$:
    \begin{gather*}
      (c_1\vee d_1)\cdot(c_2\vee d_2)=(c_1\cdot c_2)\cdot(d_1\cdot d_2)=(d_1\cdot d_2)\cdot(c_1\cdot c_2).
    \end{gather*}
    Repeating inductively the same reasoning, we finally obtain that $x\cdot y=y\cdot x\in\M_\alpha$, hence providing a common upper bound of $x$ and of $y$ in~$\M_\alpha$. This proves the existence of $x\vee y$ in $\M_\alpha$ in the case where $x\wedge y=\vd$.

    The general case follows by considering $x'=(x\wedge y)\backslash x$ and $y'=(x\wedge y)\backslash y$ instead of $x$ and~$y$.
\end{proof}

\begin{remark}
In a \DCS, for each state $\alpha\in X$, the partially ordered set of cliques $(\C_\alpha,\leq)$ identifies with the powerset $(\P(\Sigma_\alpha),\subseteq)$. In particular $\C_\alpha$ has a maximum $c_\alpha=\max(\C_\alpha)=\bigvee\Sigma_\alpha$, given by:\quad $c_\alpha=\Sigma_\alpha$. We keep this notation in the statement of the following lemma.
\end{remark}

\begin{lemma}
  \label{lem:3}
  Let $\X=(\M,X,\bot)$ be a deterministic concurrent system, and let  $\alpha\in X$. Let $T_\alpha=(c_{i})_{i\geq1}$ be the sequence of cliques defined by $c_{1}=c_\alpha$, and inductively by $c_{i+1}=c_{\alpha_{i}}$ where $\alpha_i=\alpha\cdot (c_{1}\cdot\ldots\cdot c_{i})$. Then $T_\alpha$ is a generalised execution which is the maximum of\/ $(\Mbar_\alpha,\leq)$.
\end{lemma}

\begin{proof}
  We first observe that, for $c_\alpha$ the maximum of~$\C_\alpha$, then $c_\alpha\to y$ holds\footnote{This actually holds for any concurrent system, not necessarily deterministic, if $c_\alpha$ is taken to be any maximal element in~$\C_\alpha$.} for every clique $y\in\C_{\alpha\cdot c_\alpha}$. Here in particular, $c_i\to c_{i+1}$ holds for all $i\geq1$, hence $T_\alpha$ is indeed a generalised execution.
  
  Let $x\in\Mbar_\alpha$, with $x=(d_i)_{i\geq1}$. We prove that $x\leq T_\alpha$. Assume first that $x$ is a finite trace, of height $k=\height(x)$. Put $y=c_1\cdot\ldots\cdot c_k$. Then $x$ and $y$ belong to~$\M_\alpha$. Hence $z=x\vee y$ exists in~$\M_\alpha$. Let $(e_1,\ldots,e_k)$ be the normal form of~$z$ (since $x$ and $y$ have the same height~$k$, $z$ also has height~$k$). Then $c_j\leq e_j$ and thus $c_j=e_j$ for all~$j$ by maximality of~$c_j$. Hence $d_j\leq c_j$ for all~$j$, which was to be proved.

  If $x=(c_i)_{i\geq1}$ is now a generalised trace, we obtain the same result by applying the previous case to all sub-traces $(c_i)_{1\leq i\leq k}$.
\end{proof}

Let us introduce a name for a valuation that will play a special role.

\begin{definition}
  \label{def:2}
  Let $\X=(\M,X,\bot)$ be a concurrent system. The valuation $f=(f_\alpha)_{\alpha\in X}$ defined by:
  \begin{gather*}
    \forall\alpha\in X\quad \forall x\in\M\quad f_\alpha(x)=
    \begin{cases}
      1,&\text{if $x\in\M_\alpha$}\\
      0,&\text{otherwise}
    \end{cases}
  \end{gather*}
is called the \emph{dominant valuation} of~$\X$.
\end{definition}

The family $f=(f_\alpha)_{\alpha\in X}$ given in Def.~\ref{def:2} is indeed a valuation. Indeed, using the axioms of the monoid action and the additional assumption $\bot\cdot z=\bot$ for all $z\in\M$, one sees that the following equivalence is true for every $\alpha\in X$ and for every traces $x,y\in\M$:
\begin{gather*}
  \alpha\cdot (x\cdot y)\neq\bot\iff(\alpha\cdot x\neq\bot\land (\alpha\cdot x)\cdot y\neq\bot),
\end{gather*}
which translates at once as the identity $f_\alpha(x\cdot y)=f_\alpha(x)f_{\alpha\cdot x}(y)$.

\begin{theorem}
  \label{thr:1}
  Let $\X=(\M,X,\bot)$ be a non trivial concurrent system.
  \begin{enumerate}
  \item\label{item:12} If\/ $\Sigma_\alpha\neq\emptyset$ for all $\alpha\in X$, then the two following statements are equivalent:
  \begin{compactenum}[\normalfont(i)]
  \item\label{item:1} $\X$ is deterministic.
  \item\label{item:2} The dominant valuation of $\X$ is probabilistic.
  \end{compactenum}
\item\label{item:13} If $\X$ is deterministic, then all sets~$\BM_\alpha$, for $\alpha\in X$,  are at most countable and the characteristic root of $\X$ is $r=1$ or $r=\infty$.
  \end{enumerate}
\end{theorem}

\begin{proof}
Point~\ref{item:12}. To prove the stated equivalence, assume~(i), and let $f=(f_\alpha)_{\alpha\in X}$ be the dominant valuation. Let $\alpha\in X$, and let $c\in\C_\alpha$. Since $\C_\alpha$ identifies with~$\P(\Sigma_\alpha)$, the M\"obius transform of $f_\alpha$ evaluated at $c$ is given by:
  \begin{align*}
    h_\alpha(c)&=\sum_{c'\in\C_\alpha\tqs c'\geq c}(-1)^{|c'|-|c|}=
                 \begin{cases}
                   1,&\text{if $c=c_\alpha$ (the maximum of $\C_\alpha$)}\\
                   0,&\text{otherwise}.
                 \end{cases}
  \end{align*}
  Since $\vd\neq c_\alpha$ for all $\alpha\in X$, this shows that $f$ is a probabilistic valuation.

  Conversely, assume as in~(ii) that $f$ is probabilistic. Let $\alpha\in X$ be a state, and let $c_\alpha$ be a maximal element of~$(\C_\alpha,\leq)$. Then, on the one hand, and since $c_\alpha$ is a maximal clique, one has $h_\alpha(c_\alpha)=f_\alpha(c_\alpha)=1$. But on the other hand, $h_\alpha$~is nonnegative on $\C_\alpha$ and sums up to~$1$ on~$\C_\alpha$. Hence $h_\alpha$ vanishes on all other cliques of~$\C_\alpha$. Since this is true for every maximal element of~$\C_\alpha$, it entails that $\C_\alpha$ has actually a unique maximal element, which is thus its maximum~$\Sigma_\alpha$. Hence $(\C_\alpha,\leq)$ is a lattice for every $\alpha\in X$, which proves~(i) according to Proposition~\ref{prop:1}.

  Point~\ref{item:13}. We assume that $\X$ is a \DCS. According to Lemma~\ref{lem:3}, the partial order $(\Mbar_\alpha,\leq)$ has a maximum~$T_\alpha$ for every $\alpha\in X$, hence $\Mbar_\alpha\subseteq\Mbar_{\leq T_\alpha}$.  It follows at once from Corollary~\ref{cor:1} that $\BM_\alpha$ is at most countable, and that $\#\M_\alpha(n)\leq P(n)$ for all integers~$n$ and for some polynomial~$P$. All generating series $G_{\alpha,\beta}(z)$ are rational with non zero coefficients at least~$1$, and they have their coefficients dominated by some polynomial. They have therefore a radius of convergence either $1$ or~$\infty$. Hence $r\in\{1,\infty\}$.
\end{proof}

\begin{remark}
  In general, there might exist other probabilistic valuations than the dominant valuation, even for a \DCS. See an example at the end of next section.
\end{remark}

Since the dominant valuation $f$ is probabilistic, there corresponds a family of probability measures as described in Sect.~\ref{sec:valu-prob-valu}. The behaviour of the associated Markov chain of states-and-cliques is trivial, as shown by the following result.

\begin{proposition}
  \label{prop:2}
  Let $\X=(\M,X,\bot)$ be a non trivial \DCS\ such that\/ $\Sigma_\alpha\neq\emptyset$ for all $\alpha\in X$, and let $\nu=(\nu_\alpha)_{\alpha\in X}$ be the family of probability measures associated with the dominant valuation. Then for each initial state $\alpha\in X$, the probability measure $\nu_\alpha$ is the Dirac distribution~$\delta_{\{T_\alpha\}}$, where $T_\alpha=\max\Mbar_\alpha$.
\end{proposition}

\begin{proof}
  Assuming that $\X$ is a \DCS, we keep using the notation $c_\alpha=\max\C_\alpha=\Sigma_\alpha$ for all $\alpha\in X$.
  
  A direct proof is as follows. Fix $\alpha\in X$, and let $(\alpha_{i},z_i)_{i\geq0}$ be defined inductively by $\alpha_0=\alpha$, $z_0=\vd$ and $z_{i+1}=z_i\cdot c_{\alpha_i}$, $\alpha_{i+1}=\alpha\cdot z_i$. On the one hand, we have $\bigvee_{i\geq0}z_i=T_\alpha$ by the construction used in the proof of Lemma~\ref{lem:3}. But on the other hand, the characterisation of the probability measure $\nu_\alpha$ yields $\nu_\alpha(\up z_i)=f(z_i)=1$ for all $i\geq0$. Since $\up z_{i+1}\subseteq\up z_{i}$ for all $i\geq0$, we have thus:
  \begin{gather*}
    \nu_\alpha(\omega\geq T_\alpha)=\nu_\alpha\Bigl(\,\bigcap_{i\geq0}\up z_i\Bigr)=\lim_{i\to\infty}\nu_\alpha(\up z_i)=1.
  \end{gather*}
  Since $T_\alpha=\max\Mbar_\alpha$, it implies $\nu_\alpha(\omega=T_\alpha)=1$.

An alternative proof is as follows. Let $(Y_i)_{i\geq1}$ be the Markov chain of states-and-cliques associated to the dominant valuation, and let $\alpha\in X$. One has $\nu_\alpha(C_1=c)=h_\alpha(c)$ for all $c\in\Cstar_\alpha$, by~(\ref{eq:9}). The values of $h_\alpha$ computed in the proof of Th.~\ref{thr:1} show that the initial distribution of the chain is~$\delta_{\{(\alpha,c_\alpha)\}}$. It is shown in \cite{abbes19:_markov} that the $(\alpha,c)$-row of the transition matrix of the chain is proportional to~$h_{\alpha\cdot c}(\cdot)$. Hence all entries of the  $(\alpha,c)$-row are~$0$, except for the $\bigl((\alpha,c),(\beta,c_\beta)\bigr)$ entry with $\beta=\alpha\cdot c$, where the entry is~$1$. Hence the execution $T_\alpha$ is given $\nu_\alpha$-probability~$1$.
\end{proof}

\section{Irreducible deterministic concurrent systems}
\label{sec:irred-determ-conc}

Before stating the main result of this section, we need to prove two lemmas.

\begin{lemma}
  \label{lem:2}
  Let $\X=(\M,X,\bot)$ be a \DCS. Let $\alpha\in X$ and let $c\in\C_\alpha$ be a clique such that $a\notin c$ for some letter $a\in\Sigma_\alpha$. Then:
  \begin{gather*}
    \forall x\in\Mbar_\alpha\quad C_1(x)=c\implies a\notin x.
  \end{gather*}
\end{lemma}

\begin{proof}
  Let $\alpha$, $a$ and $c$ be as in the statement. Clearly, the implication stated in the lemma is true if we prove it to be true for $x$ ranging over~$\M_\alpha$ instead of~$\Mbar_\alpha$. Hence, let $x\in\M_\alpha$ be such that $C_1(x)=c$. Let $(c_i)_{i\geq1}$ be the generalised normal form of~$x$, and define by induction $x_0=\vd$, $x_{i+1}=x_i\cdot c_{i+1}$ for all $i\geq0$ and $\alpha_i=\alpha\cdot x_i$ for all $i\geq0$. We prove by induction on $i\geq1$ that:
  \begin{inparaenum}[1)]
  \item $a\in\Sigma_{\alpha_{i-1}}$; and
  \item $a\notin c_i$.
  \end{inparaenum}

  For $i=1$, both properties derive from the assumptions of the lemma. Assume that both properties hold for some $i\geq1$. By construction, $c_i\in\C_{\alpha_{i-1}}$\,, and $a\in\Sigma_{\alpha_{i-1}}$ by the induction hypothesis. Since the concurrent system is deterministic, it follows that $a\vee c_i\in\C_{\alpha_{i-1}}$. Since $a\notin c_i$ by the assumption hypothesis, this \lub\ is given by $c_i\cdot a\in\C_{\alpha_{i-1}}$\,. This entails first that $a\in\C_{\alpha_{i-1}\cdot c_i}$\,, but $\alpha_{i-1}\cdot c_i=\alpha_i$ hence $a\in\Sigma_{\alpha_i}$\,. But it also entails that $a\notin c_{i+1}$\,, completing the induction step. The result of the lemma follows.
\end{proof}

\begin{lemma}
  \label{lem:4}
  Let $\X=(\M,X,\bot)$ be a concurrent system. Let $\alpha\in X$, and let $r_\alpha$ be the radius of convergence of the generating series $G_\alpha(z)=\sum_{x\in\M_\alpha}z^{|x|}$. Then the following properties are equivalent:
  \begin{inparaenum}[\normalfont(i)]
  \item\label{item:5} $\M_\alpha$~is finite;    
  \item\label{item:6} $\BM_\alpha=\emptyset$;
  \item\label{item:7} $r_\alpha=\infty$.
  \end{inparaenum}
\end{lemma}

\begin{proof}
  The implications $\text{(\ref{item:5}})\implies\text{(\ref{item:6}})$ and $\text{(\ref{item:5}})\implies\text{(\ref{item:7}})$ are clear. 

Assume that $\M_\alpha$ is infinite. Then there exists executions in $\M_\alpha$ of length arbitrary large. Therefore there exists $x\in\M_\alpha$ and $y\neq\vd$ such that $\alpha\cdot x=\alpha\cdot(x\cdot y)$. Then all traces $x_n=x\cdot y^n$ belong to $\M_\alpha$ for $n\geq0$. This proves two things. First, if $k=|y|$, the  coefficient of $z^{|x|+kn}$ in the series $G_\alpha(z)$ is~$\geq1$ for all integers~$n$, hence $r_\alpha<\infty$. Second, the execution $\xi=\bigvee_{n\geq0}x_n$ is an element of~$\BM_\alpha$, showing that $\BM_\alpha\neq\emptyset$. Hence we have proved both  $\text{(\ref{item:6}})\implies\text{(\ref{item:5}})$ and $\text{(\ref{item:7}})\implies\text{(\ref{item:5}})$ by contraposition, completing the proof.
\end{proof}

\begin{theorem}
  \label{thr:2}
  Let $\X=(\M,X,\bot)$ be an irreducible and non trivial concurrent system, of characteristic root~$r$, and let $f$ be the dominant valuation of~$\X$. Then the following statements are equivalent:
  \begin{compactenum}[\normalfont(i)]
  \item\label{item:3} $\X$ is deterministic.
  \item\label{item:4} $f$ is a probabilistic valuation.
  \item\label{item:8} $f$ is the only probabilistic valuation of~$\X$.
  \item\label{item:9} $r=1$.
  \item\label{item:10} One set\/ $\BM_\alpha$ is at most countable.
  \item\label{item:11} Every set\/ $\BM_\alpha$ is at most countable.
  \end{compactenum}
\end{theorem}

\begin{proof}
Since $\X$ is both irreducible and non trivial, it satisfies in particular $\Sigma_\alpha\neq\emptyset$ for all $\alpha\in X$. Hence the equivalence $\text{(\ref{item:3})}  \iff\text{(\ref{item:4})}$ and the implications $\text{(\ref{item:3})}\implies\text{(\ref{item:9})}$ and $\text{(\ref{item:3})}\implies\text{(\ref{item:11})}$ derive already from Theorem~\ref{thr:1}. The implications $\text{(\ref{item:8})}\implies\text{(\ref{item:4})}$ and $\text{(\ref{item:11})}\implies\text{(\ref{item:10})}$ are trivial.

  $\text{(\ref{item:3})}\implies\text{(\ref{item:8})}$. Let $f=(f_\alpha)_{\alpha\in X}$ be a probabilistic valuation, and let $\widetilde f=(\widetilde f_\alpha)_{\alpha\in X}$ be the dominant valuation. Let $\alpha\in X$ and let $c\in\Cstar_\alpha$ with $c\neq c_\alpha$, where $c_\alpha=\Sigma_\alpha$ is the maximum of~$\C_\alpha$. There is thus a letter $a\in\Sigma_\alpha$ such that $a\notin c$. Let $\M^a$ be the submonoid of $\M$ generated by $\Sigma\setminus\{a\}$. It follows from Lemma~\ref{lem:2} that $\{\omega\in\BM_\alpha\tq C_1(\omega)=c\}\subseteq\BM^a_\alpha$.

  According to the spectral property recalled in Section~\ref{sec:irreducibility}, the characteristic root $r^a$ of $\X^a=(\M^a,X,\bot)$ satisfies $r^a>r$ since $\X$ is assumed to be irreducible. But $r=1$ since $\X$ is deterministic, and therefore $r^a=\infty$, which implies that $\BM^a_\alpha=\emptyset$ according to Lemma~\ref{lem:4}. Let $\nu=(\nu_\alpha)_{\alpha\in X}$ be the family of probability measures associated with the probabilistic valuation~$f$, as explained in Sect.~\ref{sec:valu-prob-valu}. Then $\nu_\alpha(\BM_\alpha^a)=0$ and thus $\nu_\alpha(C_1=c)=0$. But one also has $h_\alpha(c)=\nu_\alpha(C_1=c)$ according to~(\ref{eq:9}), where $h_\alpha$ is the M\"obius transform of~$f_\alpha$. Hence $h_\alpha(c)=0$. We have proved that $h_\alpha$ vanishes on all cliques $c\in\C_\alpha$ such that $c\neq c_\alpha$. Since $(h_\alpha(c))_{c\in\Cstar_\alpha}$ is a probability vector, it entails that $h_\alpha(c_\alpha)=1$. Thus $h_\alpha$ coincides with the M\"obius transform of~$\widetilde f_\alpha$, and $f=\widetilde f$.

  $\text{(\ref{item:9})}\implies\text{(\ref{item:3})}$ and $\text{(\ref{item:10})}\implies\text{(\ref{item:3})}$ . By contraposition, assume that $\X$ is not deterministic. Prop.~\ref{prop:1} implies the existence of a state $\alpha$ and of two distinct letters $a,b\in\Sigma_\alpha$ such that $a\cdot b\neq b\cdot a$. Since $\X$ is assumed to be irreducible, there exists $x\in\M_{\alpha\cdot a,\alpha}$ and $y\in\M_{\alpha\cdot b,\alpha}$. Put $x_a=a\cdot x$ and $x_b=b\cdot y$, and we can also assume without loss of generality that $|x_a|=|x_b|$. Then $\M_\alpha$ contains the submonoid generated by $\{x_a,x_b\}$, which is free. This implies two things: first,  the generating series $G_\alpha(z)=\sum_{x\in\M_\alpha}z^{|x|}$ has radius of convergence smaller than~$1$, and thus $r<1$; second, $\BM_\alpha$~is uncountable. The proof is complete.  
\end{proof}

For an irreducible \DCS, the behaviour of the Markov chain of states-and-cliques associated to the unique probabilistic dynamics is the trivial dynamics described by Prop.~\ref{prop:2}. This is illustrated in the following example.

\begin{example}
Figure~\ref{fig:aqqqwqw} depicts an example of irreducible \DCS. The digraph of states-and-cliques of the system is depicted on Fig.~\ref{fig:aqqqapokferw}. Compare with the situation depicted next for a \DCS\ which is not irreducible.

  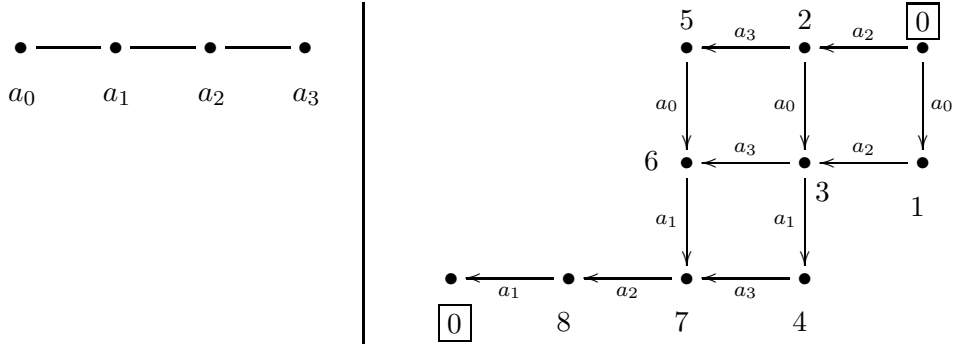
\begin{figure}
    $$
    \begin{array}{c|c}
\xymatrix{%
      \bullet\ar@{-}[r]\labeld{a_0}
&\bullet\ar@{-}[r]\labeld{a_1}                                     
&\bullet\ar@{-}[r]\labeld{a_2}                                     
&\bullet\labeld{a_3}                                     
}\quad\strut      &\qquad
      \xymatrix@R=3em@C=3em{
      &&\bullet\ar[d]_{a_0}\labelu{5}
      &\bullet\ar[d]_{a_0}\ar[l]_{a_3}\labelu{2}
      &\bullet\ar[d]^{a_0}\ar[l]_{a_2}\labelu{\fbox{$0$}}\\
      &&\bullet\ar[d]_{a_1}\labell{6}
      &\bullet\ar[d]_{a_1}\ar[l]_{a_3}\labeldr{3}
      &\bullet\ar[l]_{a_2}\labeld{1}\\
      \bullet\labeld{\fbox{$0$}}
      &\bullet\ar[l]^{a_1}\labeld{8}
      &\bullet\ar[l]^{a_2}\labeld{7}
      &\bullet\ar[l]^{a_3}\labeld{4}
      }
    \end{array}
    $$
\caption{Example of an irreducible and deterministic concurrent system $\X=(\M,X,\bot)$ with $\Sigma=\{a_0,\ldots,a_3\}$, $X=\{0,1,\ldots,8\}$. \textsl{Left:} Coxeter graph of the monoid~$\M$. \textsl{Right:} multigraph of states of~$\X$. The two framed labels $\fbox{$0$}$ are identified and correspond to the same state.}
    \label{fig:aqqqwqw}
  \end{figure}

  \begin{figure}
    $$
    \xymatrix{%
&*+[F]{\strut(5,a_0)}\ar[d]      &*+[F]{\strut(8,a_1)}\ar[r]\POS!D!R\ar@(dr,ur)[ddd]!R!U\POS[]\POS!R!D(.5)\ar[dr]!U!L
      &*+[F--]{\strut(0,a_0)}\\
*+[F]{\strut(2,a_0a_3)}\ar[r]&      *+[F]{\strut(6,a_1)}\ar[r]
      &*+[F]{\strut(7,a_2)}\ar[u]
      &*+[F--]{\strut(0,a_2)}\ar[r]
      &*+[F--]{\strut(2,a_3)}
      \\
&*+[F]{\strut(4,a_3)}\POS!U!R\ar[ur]!D!L      &*+[F]{\strut(3,a_1 a_3)}\ar[u]
      &*+[F--]{\strut(2,a_0)}\ar[d]
\\
       &*+[F]{\strut(1,a_2)}\POS!U!R\ar[ur]!D!L\POS!D!R\ar@(dr,dl)[rr]!D!L\POS[]\POS!D\ar@(d,l)[drr]
     &*+[F]{\strut(0,a_0 a_2)}\ar[u]\ar[r]\ar[dr]!U!L
      &*+[F--]{\strut(3,a_1)}
\\
&&&*+[F--]{\strut(3,a_3)}
      }
    $$
    \caption{Digraph of states-and-cliques for the \DCS\ depicted on Fig.~\ref{fig:aqqqwqw}. Nodes with solid frames are nodes of the form $(\alpha,c_\alpha)$ with $c_\alpha=\max\C_\alpha$. The probability for the Markov chain of states-and-cliques to jump from a solid frame node to a dashed frame node is~$0$; the probability of starting in a dashed node in~$0$.}
    \label{fig:aqqqapokferw}
  \end{figure}
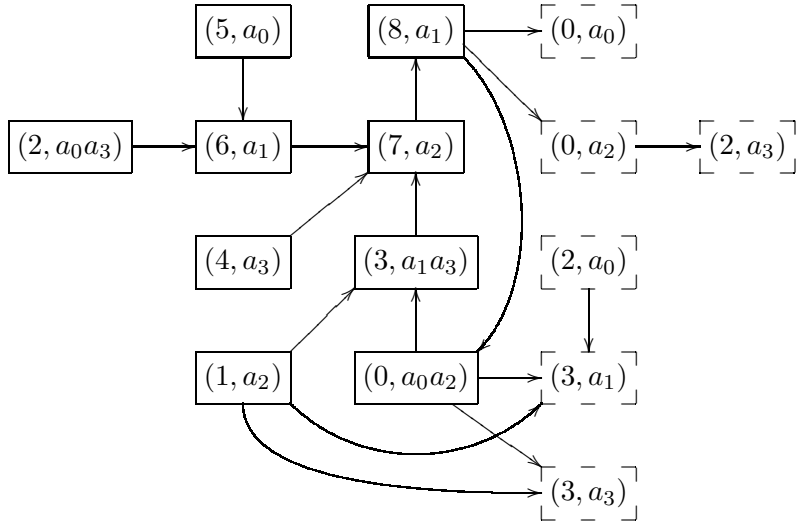
\end{example}

\begin{example}
Without the irreducibility assumption, the equivalence stated in Th.~\ref{thr:2} may fail. We give below an example of a deterministic concurrent systems not irreducible, and not satisfying point~(\ref{item:8}).

Let $\X=(\M,X,\bot)$ be the \DCS\ depicted in Fig.~\ref{fig:;lksadfk}. The system is not irreducible for several reasons: none of the three conditions for irreducibility is met. The probabilistic valuations of $\X$ are all of the following form, for some real $p\in[0,1]$:
  \begin{align*}
    f_{\alpha_0}(a)&=1&f_{\alpha_0}(c)&=p&    f_{\alpha_1}(b)&=1&f_{\alpha_1}(c)&=p&f_{\beta_0}(a)&=1&f_{\beta_1}(b)&=1
  \end{align*}

  Hence the dominant valuation is not the unique probabilistic valuation, contrary to irreducible systems as stated by point (\ref{item:8}) of Th.~\ref{thr:2}. The parameter $p$ is to be interpreted as the ``probability of playing~$c$'' in the course of the execution.  But this decision---playing $c$ or not---is made once, hence allowing all values between $0$ or~$1$ for the probability. Whereas, in a sequential model of concurrency, that would typically be a decision repeated infinitely often, hence yielding the only two possible values $0$ or $1$ for this probability.
The formula $\nu_\alpha(C_1=\gamma)=h_\alpha(\gamma)$ for $\gamma\in\C_\alpha$ yields the following initial distribution of the Markov chain of states-and-cliques if, for instance, the initial state of the system is~$\alpha_0$:
\begin{align*}
  \nu_{\alpha_0}(C_1=a)&=1-p&\nu_{\alpha_0}(C_1=c)&=0&\nu_{\alpha_0}(C_1=ac)&=p
\end{align*}

% \begin{gather*}
% \hspace{-4em}\begin{array}{c|c|c|c|c|c|}
%     \parbox{5em}{\raggedright Initial state $\alpha$}&\nu_\alpha(C_1=a)&\nu_\alpha(C_1=b)&\nu_\alpha(C_1=c)&\nu_\alpha(C_1=ac)&\nu_\alpha(C_1=bc)\\[2ex]
% \hline    \alpha_0&1-p&0&0&p&0\\
%     \alpha_1&0&1-p&0&0&p\\
%     \beta_0&1&0&0&0&0\\
%     \beta_1&0&1&0&0&0
%   \end{array}
% \end{gather*}

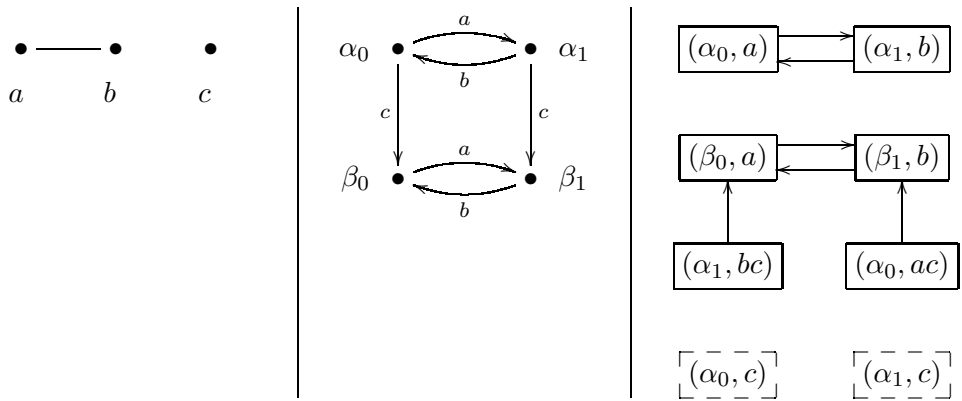
\begin{figure}
  $$\begin{array}{c|c|c}
      \xymatrix{\bullet\ar@{-}[r]\labeld{a}&\bullet\labeld{b}&\bullet\labeld{c}}
\qquad\strut&\qquad
\xymatrix@R=3.5em@C=3.5em{\bullet\ar@/^/^{a}[r]\ar[d]_ {c}\labell{\alpha_0}
&\bullet\ar@/^/^{b}[l]\ar[d]^{c}\labelr{\alpha_1}\\
\bullet\ar@/^/^{a}[r]\labell{\beta_0}&\bullet\ar@/^/^{b}[l]\labelr{\beta_1}
                                       }\qquad\strut
&\quad
\entrymodifiers={+[][F]}%
\xymatrix{
(\alpha_0,a)\ar@<1ex>[r]&(\alpha_1,b)\ar@<1ex>[l]\\
      (\beta_0,a)\ar@<1ex>[r]&(\beta_1,b)\ar@<1ex>[l]\\
      (\alpha_1,bc)\ar[u]&(\alpha_0,ac)\ar[u]\\
      *+[F--]{(\alpha_0,c)}&*+[F--]{(\alpha_1,c)}
}
  \end{array}$$
  \caption{A non irreducible \DCS\ not satisfying property (\ref{item:8}) of Th.~\ref{thr:2}. \textsl{Left:} the Coxeter graph of the monoid. \textsl{Middle:} the multigraph of states of the \DCS. \textsl{Right:} the digraph of states-and-cliques. The parameter $p$ is only involved in the initial distribution of the Markov chain of states-and-cliques. The dashed nodes are isolated in the digraph of states-and-cliques and are immaterial to the Markov chain of states-and-cliques.}
  \label{fig:;lksadfk}
\end{figure}
\end{example}

\bibliography{biblio}

\begin{thebibliography}{10}

\bibitem{abbes19:_markov}
S.~Abbes.
\newblock Markovian dynamics of concurrent systems.
\newblock {\em Discrete Event Dyn. Syst.}, 29(4):27--566, 2019.

\bibitem{abbes15}
S.~Abbes and J.~Mairesse.
\newblock Uniform and {B}ernoulli measures on the boundary of trace monoids.
\newblock {\em J. Combin. Theory Ser. A}, 135:201--236, 2015.

\bibitem{abbes20}
S.~Abbes, J.~Mairesse, and Y.-T. Chen.
\newblock A spectral property for concurrent systems and some probabilistic
  applications.
\newblock Submitted for publication. Available at
  \texttt{https://arxiv.org/abs/2003.03762}, 2020.

\bibitem{cartier69}
P.~Cartier and D.~Foata.
\newblock {\em Probl\`emes combinatoires de commutation et r\'earrangements},
  volume~85 of {\em Lecture Notes in Math.}
\newblock Springer, 1969.

\bibitem{dehornoy15}
P.~Dehornoy, F.~Digne, E.~Godelle, D.~Krammer, and J.~Michel.
\newblock {\em Foundations of Garside Theory}.
\newblock EMS, 2015.

\bibitem{diekert90}
V.~Diekert.
\newblock {\em Combinatorics on Traces}.
\newblock Springer, 1990.

\bibitem{diekert95}
V.~Diekert and G.~Rozenberg, editors.
\newblock {\em The Book of Traces}.
\newblock World Scientific, 1995.

\bibitem{goldwurm00}
M.~Goldwurm and M.~Santini.
\newblock Clique polynomials have a unique root of smallest modulus.
\newblock {\em Inform. Process. Lett.}, 75(3):127--132, 2000.

\bibitem{krob03}
D.~Krob, J.~Mairesse, and I.~Michos.
\newblock Computing the average parallelism in trace monoids.
\newblock {\em Discrete Math.}, 273:131--162, 2003.

\bibitem{nielsen81}
M.~Nielsen, Plotkin. G., and G.~Winskel.
\newblock Petri nets, event structures and domains, part~{I}.
\newblock {\em Theoret. Comput. Sci.}, 13:85--108, 1981.

\bibitem{reisig85}
W.~Reisig.
\newblock {\em Petri Nets. An Introduction}.
\newblock Springer, 1985.

\bibitem{rota64}
G.-C. Rota.
\newblock {On the foundations of combinatorial theory~I. Theory of M\"obius
  functions}.
\newblock {\em Z. Wahrscheinlichkeitstheorie}, 2:340--368, 1964.

\bibitem{viennot86}
X.~Viennot.
\newblock Heaps of pieces,~{I} : basic definitions and combinatorial lemmas.
\newblock In {\em Combinatoire \'enum\'erative}, volume 1234 of {\em Lecture
  Notes in Math.}, pages 321--350. Springer, 1986.

\end{thebibliography}
%\printbibliography
\bibliographystyle{plain}

\end{document}